\documentclass{article}

\usepackage{
  amsmath,
  amssymb,
  amsthm,
  hyperref,
  paralist,
  geometry,
  tikz,
  thmtools,
  mathpazo
}
\usetikzlibrary{arrows}
\usetikzlibrary{positioning,shapes}

\ifcsname theorem\endcsname{}\else\declaretheorem[parent=section]{theorem}\fi
\ifcsname corollary\endcsname{}\else\declaretheorem[sibling=theorem]{corollary}\fi
\ifcsname lemma\endcsname{}\else\declaretheorem[sibling=theorem]{lemma}\fi
\ifcsname proposition\endcsname{}\else\declaretheorem[sibling=theorem]{proposition}\fi
\ifcsname conjecture\endcsname{}\else\declaretheorem[sibling=theorem]{conjecture}\fi
\ifcsname problem\endcsname{}\else\fi
\ifcsname question\endcsname{}\else\fi
\ifcsname definition\endcsname{}\else\declaretheorem[sibling=theorem, style=definition]{definition}\fi
\ifcsname exercise\endcsname{}\else\fi
\ifcsname example\endcsname{}\else\declaretheorem[sibling=theorem, style=definition]{example}\fi
\ifcsname remark\endcsname{}\else\declaretheorem[sibling=theorem, style=remark]{remark}\fi

\providecommand {\Q}{{\bf Q}}

\providecommand {\C}{{\bf C}}

\renewcommand {\P}{{\bf P}}

\providecommand{\PGL}{\operatorname{PGL}}

\providecommand {\from}{{\colon}}

\providecommand{\spec}{\operatorname{Spec}}
\providecommand{\proj}{\operatorname{Proj}}

\providecommand{\Hom}{\operatorname{Hom}}
\providecommand{\Ext}{\operatorname{Ext}}

\providecommand{\Aut}{\operatorname{Aut}}
\providecommand{\codim}{\operatorname{codim}}
\providecommand{\Pic}{\operatorname{Pic}}
\providecommand{\Sym}{\operatorname{Sym}}
\providecommand{\rk}{\operatorname{rk}}

\renewcommand{\H}{\mathcal H}
\newcommand{\V}{\mathcal V}
\newcommand{\U}{\mathcal U}
\newcommand{\M}{\mathcal M}
\newcommand{\td}{\widetilde}
\newcommand{\F}{{\mathbf F}}
\newcommand{\sub}{{\rm sub}}
\newcommand{\gen}{{\rm gen}}
\newcommand{\irr}{{\rm irr}}
\DeclareMathOperator{\Br}{Br}
\DeclareMathOperator{\Ram}{Ram}

\numberwithin{equation}{section}

\declaretheorem[unnumbered]{claim}
\declaretheorem[title=Theorem, style=theorem]{maintheorem}

\title{The Picard rank conjecture for the Hurwitz spaces of degree up to five}
\author{Anand Deopurkar \and Anand Patel}

\begin{document}
\maketitle
\begin{abstract}
  We prove that the rational Picard group of the simple Hurwitz space $\mathcal H_{d,g}$ is trivial for $d$ up to five. We also relate the rational Picard groups of the Hurwitz spaces to the rational Picard groups of the Severi varieties of nodal curves on Hirzebruch surfaces.
\end{abstract}

%\tableofcontents
\section{Introduction}\label{sec:introduction}
Let $\H_{d,g}$ be the simple Hurwitz space which parametrizes isomorphism classes of simply branched degree $d$ covers of genus zero curves by genus $g$ curves. Although $\H_{d,g}$ has been studied classically, many fundamental questions about its geometry are still unanswered. The goal of this paper is to address one such question, the question of its Picard group. It is conjectured (for example, \cite{diaz96:_towar_homol_of_hurwit_spaces}) that the rational Picard group $\Pic_\Q(\H_{d,g})$ is trivial. We call this the \emph{Picard rank conjecture} for $\H_{d,g}$. Our main result is a proof of this conjecture for $d \leq 5$.
\begin{maintheorem}\label{thm:prc345}
  The rational Picard group of $\H_{d,g}$ is trivial for $d \leq 5$.
\end{maintheorem}
In the main text, \autoref{thm:prc345} is divided into the case of degree 3 (\autoref{thm:PRC3}), degree 4 (\autoref{thm:PRC4}),  and degree 5 (\autoref{thm:PRC5}).

The Picard rank conjecture was known for $d = 2$ and $3$. For $d = 2$, it was proved by Cornalba and Harris \cite[Lemma~4.5]{cornalba88:_divis}, and for $d = 3$ by Stankova-Frenkel \cite[\S~12.2]{Stankova-Frenkel00:_Modul_Of_Trigon_Curves}. In these cases, now there are more refined results about the Picard group of the moduli stacks; see \cite{cornalba:pic_hyper} for $d = 2$ and \cite{bolognesi12:_stack} for $d = 3$.

The conjecture is also known for $d > 2g-2$. In this range, the map $\H_{d,g} \to \M_g$ is a fibration, where $\M_g$ is the moduli space of smooth curves of genus $g$. An analysis of this fibration shows that $\Pic_\Q(\H_{d,g}) = 0$ if and only if $\Pic_\Q(\M_g) \cong \Q$ (see, for example, \cite{mochizuki95:_hurwit} or \cite[\S~3]{diaz96:_towar_homol_of_hurwit_spaces}). Thus, the conjecture for $d > 2g-2$ follows from Harer's theorem \cite{harer:Pic_Mg}.

We briefly explain the rationale behind the conjecture. Let us blur the distinction between the coarse moduli spaces and the fine moduli stacks. This is harmless, since we are concerned with the rational Picard group. Let us also take $d \geq 4$ (the discussion holds for $d = 2, 3$ with minor modifications). Denote by $\td \H_{d,g}$ the partial compactification of $\H_{d,g}$ that parametrizes covers $[\alpha \from C \to \P^1]$ where $C$ is allowed to be nodal, but still irreducible, and $\alpha$ need not be simply branched. Let $\alpha \from \mathcal C \to \mathcal P$ be the universal family over $\td \H_{d,g}$, where $\rho \from \mathcal C \to \td \H_{d,g}$ is a family of irreducible, at worst nodal curves of arithmetic genus $g$, and $\pi \from \mathcal P \to \td \H_{d,g}$ a family of smooth curves of genus $0$. From this data, we can construct three `tautological' divisor classes on $\td \H_{d,g}$ given by 
\[\rho_*(c_1(\omega_\rho)^2), \quad \rho_*(c_1(\omega_\rho) \alpha^*c_1(\omega_\pi)), \text{ and } \rho_*([\delta_\rho]).\]
Here $\omega$ stands for the relative dualizing sheaf and $\delta$ for the singular locus. It is easy to check that the three tautological classes are $\Q$-linearly independent. On the other hand, $\td \H_{d,g} \setminus \H_{d,g}$ is a union of three irreducible divisors, namely the locus $\Delta$ where $C$ is singular, the locus $T$ where $\alpha$ has a higher order ramification point, and the locus $D$ where $\alpha$ has two ramification points over a branch point. It is also easy to check that the classes of $\Delta$, $T$, and $D$ are $\Q$-linearly independent. Thus, $\Pic_\Q(\H_{d,g}) = 0$ is equivalent to $\Pic_\Q(\td \H_{d,g})$ being generated by the tautological classes. The Picard rank conjecture thus expresses the often-satisfied expectation that there are no other divisor classes than the tautological ones.

We now outline our strategy for proving \autoref{thm:prc345}. Let $\alpha \from C \to \P^1$ be a degree $d$ cover. Then $C$ embeds in a $\P^{d-2}$-bundle over $\P^1$, which we denote by $\P E \to \P^1$. Thanks to the work of Casnati and Ekedahl, the resolution of the ideal of $C$ in $\P E$ can be described explicitly. The terms in this resolution involve (twists of) vector bundles on $\P^1$ \cite{casnati96:_cover_of_algeb_variet_i}. Let $U \subset \td \H_{d,g}$ be the open locus where these vector bundles are the most generic. The key steps in our proof are the following.
\begin{enumerate}
\item Identify the divisorial components of $\td \H_{d,g} \setminus U$.
\item Express $U$ as a (successive) quotient of an open subset of an affine space by actions of linear algebraic groups.
\item Use the previous two steps to get a bound on the Picard rank of $\td \H_{d,g}$, and in turn, the Picard rank of $\H_{d,g}$.
\end{enumerate}
Needless to say, we are able to carry out all three steps only for $d
\leq 5$. However, we can carry out parts of step (1) in general. For
step (2), we highlight that it remains unknown in general whether one
can dominate $\td \H_{d,g}$ by an affine space for $d \geq 6$.

To analyze $\td \H_{d,g} \setminus U$, we must analyze the loci in
$\td \H_{d,g}$ where the bundle $E$ and the vector bundles appearing
in the resolution of $C$ are unbalanced. We call these loci the Maroni
loci and the Casnati--Ekedahl loci, respectively. We spend significant
effort on understanding the decomposition of $\td \H_{d,g}$ into these
loci. Contained in \autoref{sec:CE}, the results of this analysis may
be of independent interest.

A key tool in our analysis is a construction that relates the Maroni
loci to the Severi varieties of Hirzebruch surfaces. Originally due to
Ohbuchi \cite{ohbuchi:relations}, this `associated scroll
construction' allows us to get the required dimension estimates. The
key input here is a theorem of Tyomkin that guarantees that the Severi
varieties of Hirzebruch surfaces are irreducible of the expected
dimension \cite{tyomkin:irred}.

The associated scroll construction also lets us relate the Picard
ranks of the Hurwitz spaces to the Picard ranks of the Severi
varieties. To state our result, let us denote by $\U_g(\F_m, d\tau)$ the
space of irreducible nodal curves of geometric genus $g$ in the linear
system $|d\tau|$ on the Hirzebruch surface $\F_m$, where $\tau$ is the
section with self-intersection $m$.
\begin{maintheorem}\label{thm:B}
  Let $m \geq \lfloor(g+d-1)/(d-1)\rfloor$. Then $\Pic_\Q \U_g(\F_m, d\tau)
  = 0$ implies $\Pic_\Q\H_{d,g} = 0$.

  Let $m \geq \lceil 2(g+d-1)/(d-1) \rceil$. Then
  $\Pic_\Q\U_g(\F_m, d\tau) = 0$ if and only if $\Pic_\Q\H_{d,g} = 0$.
\end{maintheorem}
In the main text, \autoref{thm:B} is \autoref{SvsH}.

\subsection{Notation}\label{sec:notation}
We work with a few different versions of the Hurwitz spaces. We assemble their definitions here. We work over the field $\C$ of complex numbers. By a curve, we mean a connected, proper, reduced scheme of finite type over $\C$. Throughout, assume that $g \geq 3$.
\begin{description}
\item [$\H_{d,g}$:] This is the coarse moduli space of $[\alpha \from C \to \P^1]$, where $C$ is a smooth curve of genus $g$ and $\alpha$ a finite map of degree $d$ with simple branching (that is, the branch divisor of $\alpha$ is supported at $2g+2d-2$ distinct points). Two such covers $[\alpha_1 \from C_1 \to \P^1]$ and $[\alpha_2 \from C_2 \to \P^1]$ are considered isomorphic if there are isomorphisms $\phi \from C_1 \to C_2$ and $\psi \from \P^1 \to \P^1$ such that $\alpha_2 \circ \phi = \psi \circ \alpha_1$.

\item [$\td \H_{d,g}$:] This is the coarse moduli space of $[\alpha \from C \to \P^1]$, where $C$ is an irreducible, at worst nodal curve of arithmetic genus $g$, and $\alpha$ a finite map of degree $d$. The isomorphism condition is the same as that for $\H_{d,g}$.

\item [$\H_{d,g}^\dagger$:] This is like $\H_{d,g}$, but with `framed' target $\P^1$. The objects it parametrizes are $[\alpha \from C \to \P^1]$ as in the description of $\H_{d,g}$, but $[\alpha_1 \from C_1 \to \P^1]$ and $[\alpha_2 \from C_2 \to \P^1]$ are considered isomorphic if there is an isomorphism $\phi \from C_1 \to C_2$ such that $\alpha_2 \circ \phi = \alpha_1$.

\item [$\td \H_{d,g}^\dagger$:] This is like $\td \H_{d,g}$, but with framed target $\P^1$.
\end{description}
All four are irreducible quasi-projective varieties with at worst quotient singularities. In particular, they are normal and $\Q$-factorial. The group $\Aut \P^1 = \PGL_2$ acts on the framed versions. The unframed versions are the quotients by this action in the sense that the fibers of the morphism from the framed space to the unframed space are precisely the $\PGL_2$ orbits. We have 
\[\dim \H_{d,g} = \dim \td \H_{d,g} = 2g+2d-5,\]
and 
\[\dim \H_{d,g}^\dagger = \dim \td \H_{d,g}^\dagger = 2g+2d-2.\]

In addition, we work with the following Severi varieties:
\begin{description}
\item[$\U_{g}(\F_m, d\tau)$:] This is the locus of irreducible nodal curves of geometric genus $g$ in the linear series $|d\tau|$ in the Hirzebruch surface $\F_m$. Here $\tau \subset \F_m$ is the section of self-intersection $m$.
\item[$\V_{g}(\F_m, d\tau)$:] This is the closure of $\U_{g}(\F_m, d\tau)$ in the projective space $|d\tau|$.
\item[$\V^\irr_{g}(\F_m, d\tau)$:] This is the open subset of reduced and irreducible curves in $\V_{g}(\F_m, d\tau)$.
\end{description}

We do not distinguish between a vector bundle and the corresponding locally free sheaf. Note that the vector bundle associated to the locally free sheaf $F$ is the relative $\spec$ of the symmetric algebra on $F^\vee$.

\section{Preliminaries}\label{sec:preliminaries}
In this expository section, we recall two key results. The first describes the Picard group of the quotient of a variety by a group action. The second is a structure theorem for finite covers which enables us to describe a large open subset of the Hurwitz space as such a quotient.

\subsection{Picard groups of quotients}
Let $G$ be a linear algebraic group acting on a variety $X$. Denote by $\Pic_GX$ the group of $G$-linearized line bundles on $X$. Forgetting the $G$-linearization gives a homomorphism $\Pic_GX \to \Pic X$.
\begin{proposition}\label{thm:kkv}\cite[Lemma~2.2 + Proposition~2.3]{hanspeter89:_picar_group_of_g_variet} For a connected linear algebraic group $G$ acting on an irreducible variety $X$, we have an exact sequence
   \[ \chi(G) \to \Pic_G X \to \Pic X,\]
   where $\chi(G)$ is the group of (algebraic) characters of $G$. Furthermore, if $X$ is normal, then the sequence has an extension by a homomorphism $\Pic X \to \Pic G$.
 \end{proposition}
 
 Let $\pi \from X \to Y$ be a morphism that is equivariant with the trivial $G$ action on $Y$. Let $L$ be a line bundle on $Y$. The pullback $\pi^*L$ carries a natural $G$-linearization. We thus have a homomorphism $\Pic Y \to \Pic_G X$.
 \begin{proposition}\label{thm:pic_quotient}
   Let $X$ and $Y$ be irreducible normal varieties, $G$ a linear algebraic group acting on $X$, and $\pi \from X \to Y$ a surjective morphism, equivariant with the trivial action on $Y$. Suppose the fibers of $\pi$ consist of single $G$-orbits. Then the map $\Pic Y \to \Pic_G X$ is injective and we have
   \[ \rk \Pic Y  \leq \rk \chi(G) + \rk \Pic X.\]
   Furthermore, if $G$ is reductive and the stabilizers $G_x$ are finite, then we have an isomorphism
   \[ \Pic Y \otimes \Q \xrightarrow{\sim} \Pic_G X \otimes \Q.\]
\end{proposition}
 \begin{proof}
   Suppose $L$ is a line bundle on $Y$ such that $\pi^*L$ is trivial as a $G$-linearized line bundle. Then $\pi^*L$ has a $G$-invariant nowhere-vanishing section. We claim that such a section descends to a nowhere-vanishing section of $L$ on $Y$. The crucial point is that in our setup, $Y$ is a geometric quotient of $X$ \cite[Proposition~0.2]{GIT}. That is, for every open $U \subset Y$, the preimage $\pi^{-1}U$ is open and the functions on $U$ are the invariant functions on $\pi^{-1}U$:
   \[ \Gamma(U, O_Y) = \Gamma(\pi^{-1} U, O_X)^G.\]
   It follows that the sections of $L$ on $U$ are the invariant sections of $\pi^*L$ on $\pi^{-1}(U)$:
   \[ \Gamma(U, L) = \Gamma(\pi^{-1}U, \pi^* L)^G.\]
   Thus, a $G$-invariant section $\sigma$ of $\pi^*L$ on $X$ gives a section $\overline \sigma$ of $L$ on $Y$. It is easy to check that if $\sigma$ is nowhere-vanishing, so is $\overline \sigma$.

   The bound on $\rk \Pic Y$ follows from the injectivity and \autoref{thm:kkv}. For the last statement, we use the characterization of the image of $\Pic Y \to \Pic_GX$ from \cite[Proposition~4.2]{hanspeter89:_picar_group_of_g_variet}: a $G$-linearized line bundle $L$ is in the image if and only if for every $x \in X$, the stabilizer group $G_x$ acts trivially on the fiber $L_x$. Since the stabilizers are finite, we can arrange this by passing to a large enough power of $L$.
 \end{proof}

 We end with a simple application.
\begin{proposition}\label{thm:framed_unframed}
  Let $U \subset \td \H_{d,g}$ be any open subset and $U^\dagger$ its preimage under $\td \H_{d,g}^\dagger \to \td \H_{d,g}$. Then 
  \[ \rk \Pic U = \rk \Pic U^\dagger.\]
\end{proposition}
\begin{proof}
  Apply \autoref{thm:kkv} and \autoref{thm:pic_quotient} with $G = \PGL_2$, $X = U^\dagger$, and $Y = U$.
\end{proof}

 \subsection{The Casnati--Ekedahl structure theorem}
 Let $X$ and $Y$ be integral schemes and $\alpha \from X \to Y$ a finite flat Gorenstein morphism of degree $d \geq 3$. The map $\alpha$ gives an exact sequence
\begin{equation}\label{structure sheaf sequence}
  0 \to O_Y \to \alpha_* O_X \to {E_\alpha}^\vee \to 0,
\end{equation}
where $E = E_\alpha$ is a vector bundle of rank $(d-1)$ on $Y$, called the \emph{Tschirnhausen bundle} of $\alpha$. Denote by $\omega_\alpha$ the dualizing sheaf of $\alpha$. Applying $\Hom_Y(-,O_Y)$ to \eqref{structure sheaf sequence}, we get
\begin{equation}\label{dual structure sheaf sequence}
  0 \to E \to \alpha_* \omega_\alpha \to O_Y \to 0.
\end{equation}
The map $E \to \alpha_* \omega_\alpha$ induces a map $\alpha^* E \to \omega_\alpha$. 

\begin{theorem}\label{thm:CE}
  \cite[Theorem~2.1]{casnati96:_cover_of_algeb_variet_i}
  In the above setup, $\alpha^* E \to \omega_\alpha$ gives an embedding $\iota \from X \to \P E$ with $\alpha  = \pi \circ \iota$, where $\pi \from \P E \to Y$ is the projection. Moreover, the subscheme $X \subset \P E$ can be described as follows.
\begin{enumerate}
\item The resolution of $O_X$ as an $O_{\P E}$ module has the form
\begin{equation}\label{eqn:casnati_resolution}
  \begin{split}
    0 \to \pi^* N_{d-2} (-d) \to \pi^* N_{d-3}(-d+2) \to \pi^*N_{d-4}(-d+3) \to \dots \\
    \dots \to \pi^*N_2(-3) \to \pi^*N_1(-2) \to O_{\P E} \to O_X \to 0,
  \end{split}
\end{equation}
where the $N_i$ are vector bundles on $Y$. Restricted to a point $y \in Y$, this sequence is the minimal free resolution of $X_y \subset \P E_y$.
\item The ranks of the $N_i$ are given by
  \[ \rk N_i = \frac{i(d-2-i)}{d-1} {d \choose {i+1}},\]
\item We have $N_{d-2} \cong \pi^* \det E$. Furthermore, the resolution is symmetric, that is, isomorphic to the resolution obtained by applying $\Hom_{O_{\P E}}(-, N_{d-2}(-d))$.
\end{enumerate}
\end{theorem}

The branch divisor of $\alpha \from X \to Y$ is given by a section of $(\det E)^{\otimes 2}$. In particular, if $X$ is a curve of (arithmetic) genus $g$, $\alpha$ has degree $d$, and $Y = \P^1$, then 
\begin{equation}\label{eqn:rk_deg_E}
 \rk E = d-1 \text{ and } \deg E = g+d-1.
\end{equation}

\section{The Maroni and Casnati--Ekedahl loci}\label{sec:CE}
Consider a cover $\alpha \from C \to \P^1$ and the relative canonical embedding $C \subset \P E_\alpha$. Since vector bundles on $\P^1$ split as direct sums of line bundles, the vector bundle $E_\alpha$, and the higher syzygy bundles $N_i$ appearing in \autoref{thm:CE} are discrete invariants of $\alpha$. We thus get a decomposition of the Hurwitz space into locally closed subsets where the isomorphism type of the bundles $E_\alpha$ and $N_i$ are constant. This section is devoted to the analysis of some of these locally closed subvarieties, particularly their dimensions. We only consider the bundle $E_\alpha$ and $F_\alpha := N_1$. Note that
\[ E_\alpha = \ker(\alpha_* \omega_\alpha \to O_Y) \text{ and } F_\alpha = \alpha_* I_C(2),\]
where $I_C \subset O_{\P E_\alpha}$ is the ideal sheaf of $C$. 
\begin{definition}\label{def:maroni_ce}
  For vector bundles $E$ and $F$ on $\P^1$, define the following closed subvarieties of $\H_{d,g}^\dagger$:
  \begin{align*}
    M(E, F) &:= \overline{\{ [\alpha \from C \to \P^1] \mid E_\alpha \cong E \text{ and } F_\alpha \cong F\}},\\
    M(E) &:= \overline {\{ [\alpha \from C \to \P^1] \mid E_\alpha \cong E\}},\\
    C(F) &:= \overline {\{ [\alpha \from C \to \P^1] \mid F_\alpha \cong F\}}.
  \end{align*}
  Call $M(E)$ the \emph{Maroni loci} and $C(F)$ the \emph{Casnati--Ekedahl loci}. Define subvarieties $\td M(E, F)$, $\td M(E)$, and $\td C(F)$ of $\td \H_{d,g}^\dagger$ analogously. 
\end{definition}
Abusing notation, we denote the images of these loci in the unframed versions $\H_{d,g}$ and $\td \H_{d,g}$ by the same letters. The framed versus unframed setting is usually clear by context, and sometimes irrelevant, for example in discussing the codimensions. We caution the reader that these loci are not necessarily irreducible or of expected dimension (\autoref{ex:1}, \autoref{ex:2}). Even determining whether they are non-empty remains a challenge in full generality.

\subsection{The associated scroll construction}\label{sec:ass_scroll}
To analyze the Maroni loci $M(E)$, we associate to a cover of $\P^1$ a curve on a Hirzebruch surface. The construction is originally due to Ohbuchi \cite{ohbuchi:relations}. Let $C$ be an irreducible curve of arithmetic genus $g$ and $\alpha \from C \to \P^1$ a finite cover of degree $d$. Let $\zeta$ be a global section of $O_C(m) = \alpha^*O_{\P^1}(m)$ that projects to a nonzero section of $E^\vee_\alpha(m)$. In other words, $\zeta$ is not a pullback of a section from $\P^1$. The section $\zeta$ gives a map from $C$ to the total space of the line bundle $O(m)$ over $\P^1$. Let $\F_m = \proj (O \oplus O(-m))$ be the Hirzebruch surface that compactifies this total space. We thus get the diagram
\[
\begin{tikzpicture}[node distance=1.5cm]
  \node (C) {$C$};
  \node [right of =C] (F) {$\F_m$};
  \node [below of = F] (P) {$\P^1$};
  \draw [->] 
  (C) edge node [below left] {$\alpha$} (P) 
  (C) edge node[above] {$\nu$} (F) 
  (F) edge node [right] {$\pi$} (P);
\end{tikzpicture}.
\]
Let $\sigma \subset \F_m$ be the directrix and $\tau \subset \F_m$ the section disjoint from $\sigma$ (so that $\sigma^2 = -m$ and $\tau^2 = m$). By construction, $\nu(C) \subset \F_m$ avoids the directrix $\sigma$. Suppose $C$ is smooth and $\alpha \from C \to \P^1$ does not factor nontrivially. Then $\nu$ is birational onto its image, and therefore $\nu(C)$ is a reduced and irreducible element of the linear system $|d\tau|$. By the following proposition, $\nu(C)$ is a point in the Severi variety $\V_g(\F_m,d\tau)$. 

\begin{proposition}
  A reduced and irreducible curve on $\F_m$ of geometric genus $g$ in the linear system $|d\tau|$ is a flat limit of irreducible nodal curves of geometric genus $g$.
\end{proposition}
\begin{proof}
  Let $\overline C \subset \F_m$ be such a reduced and irreducible curve. Let $C \to \overline C$ the normalization. Denote by $\nu$ the composite map $\nu \from C \to \F_m$. Let $\mathcal M$ be a component of the Kontsevich space of maps $\mathcal M_{g}(\F_m, d\tau)$ containing $\nu$. Let $N_\nu$ be the normal sheaf of $\nu$; this is the cokernel of $T_C \to \nu^* T_{\F_m}$. Then, we have a lower bound: $\dim \mathcal M \geq \chi(N_\nu)$. Since
  \[ \chi(N_\nu) = \chi(\nu^* T_{\F_m}) - \chi(T_C) = g - \deg(K_{\F_m} \cdot \overline C) - 1,\]
  we get 
  \[\dim \mathcal M \geq g - \deg(K_{\F_m} \cdot \overline C) - 1.\]
  By \cite[Proposition~2.2]{joe:severi}, a general $\nu_{\rm gen} \from C_{\rm gen} \to \F_m$ in $\mathcal M$ is birational onto its image and the image has only nodes as singularities.
\end{proof}

We can make the construction in a family. Let $M$ be a reduced scheme, $\rho \from C \to M$ a generically smooth family of reduced and irreducible curves of genus $g$, and $\alpha \from C \to \P^1 \times M$ a finite flat $M$-morphism of degree $d$. Set $O_C(m) = \alpha^*O(m)$. Assume that none of the fibers $\alpha_t \from C_t \to \P^1$ factor nontrivially and $H^0(C_t, O_{C_t}(m))$ has constant rank. Then $\rho_*O_C(m)$ is a vector bundle on $M$. The trivial subbundle $H^0(\P^1, O(m)) \otimes O_M$ maps injectively to $\rho_*O_C(m)$. Let $U$ be the complement of the image of this map in the total space of $\rho_* O_C(m)$. Fiberwise, the sections of $U$ correspond to the sections $\zeta$ which project nontrivially onto $E_\alpha^\vee(m)$. Then the associated scroll construction gives a morphism 
\[U \to \V_g(\F_m,d\tau).\]
We will use this construction where $M$ is a Maroni locus. As described, the construction depends on the existence of a universal family, and thus gives a morphism from the fine moduli stack. But since $\V_g(\F_m,d\tau)$ is a scheme, we get a canonical induced map from the coarse space.

The following crucial result makes the above construction useful.
\begin{theorem}[\cite{tyomkin:irred}]\label{tyomkin:irred}
  All Severi varieties parametrizing irreducible curves on Hirzebruch surfaces are irreducible and of expected dimension. In particular, the variety $\V_g(\F_m,d\tau)$ is irreducible of dimension $dm+2d+g-1$.
\end{theorem}

We also need the following result, which we prove for the lack of a reference.
\begin{proposition}\label{thm:severi_simply_branched}
  Let $\overline C \subset \F_m$ be a general point of $\V_g(\F_m,d\tau)$ and $C \to \overline C$ the normalization. Then the composite $C \to \P^1$ is simply branched.
\end{proposition}
\begin{proof}
  In light of \autoref{tyomkin:irred}, it suffices to exhibit a particular $\overline C$ of geometric genus $g$ in $\V_g(\F_m,d\tau)$ whose normalization is simply branched over $\P^1$. One way is to start with $X = \P^1$ and $\alpha \from X \to \P^1$ a simply branched cover of degree $d$. Then $E_\alpha = O(1)^{\oplus (d-1)}$. Choosing a general section of $E^\vee_\alpha(m)$ gives $\nu \from X \to \F_m$ such that $\nu(X)$ is nodal. It is easy to see that $\nu(X)$ is in the closure of $\V_g(\F_m,d\tau)$. Indeed, since the set of nodes of $\nu(X)$ impose independent conditions on $|K_{\F_m} + d\tau|$, they automatically impose independent conditions on $|d\tau|$ as well, and hence we may smooth out the required number of nodes of $\nu(X)$ to deform to a curve of geometric genus $g$. A general fiber of such a smoothing is the required $\overline C$.
\end{proof}

\begin{remark}\label{geometricconst} We can realize the associated scroll construction geometrically as follows.  The choice of a general global section $\zeta$ of $O_{C}(m)$ can  be thought of as a choice of a geometric section $\sigma \from \P^1 \to \P E$.  In the $\P^{d-2}$ fibers of $\pi \from  \P E  \to \P^1$, we now have $d+1$ points: $d$ points coming from the fibers of the map $\alpha \from C \to \P^1$, and one more point provided by the section $\sigma$. For general $t \in \P^1$, these $d+1$ points will be in general position, and so will define a unique rational normal curve $R_{t} \subset \P E$. Consider the birationally ruled surface $S \subset \P E$ defined as the closure of the union of the $R_{t}$'s.  $S$ contains both $\sigma$ and $C$, and is fibered over $\P^1$.  We contract all components of the fibers of the projection $\pi \from S \to \P^1$ which do not meet the directrix $\sigma$.  The resulting surface is $\F_{m}$, with $\sigma$ being the directrix. The image of $C$ under the contraction $S \to \F_{m}$ is the associated scroll construction. 

\end{remark}

For a vector bundle $E = O(a_1) \oplus \dots \oplus O(a_n)$ on $\P^1$, set
\[ \lfloor E \rfloor = \min\{a_i\} \text{ and } \lceil E \rceil = \max\{a_i\}.\]
Given a cover $\alpha \from C \to \P^1$, the associated scroll construction $\nu \from C \to \F_m$ can be made for $m \geq \lfloor E_\alpha \rfloor$. Conversely, given a point $\overline C \in \V_g^\irr(\F_m,d\tau)$, let $C \to \overline C$ be the normalization. Then the induced cover $\alpha \from C \to \P^1$ has $\lfloor E_\alpha \rfloor \leq m$.

\begin{proposition}\label{thm:constraintsE} 
  If $\td M(E)$ is nonempty, then 
  \begin{equation}\label{eqn:highsummand}
    \lceil E \rceil \leq \frac{2g+2d-2}{d}.
  \end{equation}
  Furthermore, if $E_{\alpha}$ comes from a cover $[\alpha \from C \to \P^1]$, with $C$ irreducible, and where $\alpha$ does not factor nontrivially, then
  \begin{equation}\label{eqn:lowsummand}
    \frac{g+d-1}{{d \choose 2}} \leq \lfloor E_{\alpha} \rfloor \leq \frac{g+d-1}{d-1}.
  \end{equation}
\end{proposition}
\begin{proof}
The resolution of $O_C$ in \autoref{thm:CE} tells us that $C \subset \P E_{\alpha}$ is not contained in any hyperplane divisor. Let $h$ denote the hyperplane divisor class associated to $O_{\P E_{\alpha}}(1)$, and let $f$ denote the class of the fiber of $\pi \from \P E \to \P^1$. Set $N := \lceil E_{\alpha} \rceil$.  Then the divisor class $h-Nf$ is effective.  Since $C$ is irreducible and does not lie in $(h-Nf)$, it intersects $(h-Nf)$ non-negatively. Since $h \cdot [C] = 2g+2d-2$, and $f \cdot [C] = d$, we conclude that $N \leq \frac{2d+2g-2}{d}$. 

For the second inequality, we appeal to the associated scroll construction.  Let $n := \lfloor E_{\alpha} \rfloor$. Since $\alpha$ does not factor, $\nu \from C \to \F_n$ must be birational onto its image. Adjunction on $\F_n$ gives
\[p_{a}(\nu(C)) = {d \choose 2}n - (d-1).\]
The second statement now follows from the inequality $g \leq p_{a}(\nu(C))$.
\end{proof}

The following theorem of Ohbuchi \cite{ohbuchi:relations} places a strong restriction on a large class of Tschirnhausen bundles $E$.
\begin{proposition}[\cite{ohbuchi:relations}]
  \label{thm:ohbuchi}
  Let $\alpha \from C \to \P^1$ be a cover of degree $d$, with $C$ irreducible, and where $\alpha$ does not factor nontrivially. Write $E_{\alpha} = O(a_1) \oplus \dots \oplus O(a_{d-1})$ where $\lfloor E_\alpha \rfloor =  a_1 \leq a_2 \leq \dots \leq a_{d-1} = \lceil E_\alpha \rceil$. 
  Then 
  \begin{equation}\label{eqn:tame}
a_{i+1} - a_{i} \leq \lfloor E_\alpha \rfloor \quad \text{for $1 \leq i \leq d-2$}.
\end{equation}
\end{proposition}

\begin{remark}
\autoref{thm:ohbuchi} implies the second inequality in \autoref{thm:constraintsE}.
\end{remark}

\begin{definition}
  We call a vector bundle $E$ on $\P^1$ of rank $d-1$ and degree $g+d-1$ \emph{tame} if it satisfies the inequalities \eqref{eqn:highsummand}, \eqref{eqn:lowsummand}, and \eqref{eqn:tame}.
\end{definition}
Notice that \autoref{thm:constraintsE} and \autoref{thm:ohbuchi} imply that $E_\alpha$ is tame in the following two cases: $\alpha$ is simply branched, or $d$ is prime. Indeed, in either case, the cover cannot factor non-trivially.

Denote by $\leadsto$ the partial order on vector bundles on $\P^1$ given by $E \leadsto E'$ if $E$ specializes to $E'$ in a flat family. Define the finite set ${\mathcal T}[m]$ by
\[ {\mathcal T}[m] := \{\text{Isomorphism classes of tame bundles $E$ of rank $d-1$, degree $g+d-1$, and $\lfloor E \rfloor = m$}\}.\]
Observe that $\mathcal T[m]$ contains an element $E[m]$ such that $E[m] \leadsto E$ for all $E \in \mathcal T[m]$. In other words, $E[m]$ is the most generic among all the bundles in $\mathcal T[m]$.
\begin{theorem}\label{thm:maroniloci}
  Let $m$ be an integer satisfying $\frac{g+d-1}{{d \choose 2}} \leq m \leq  \frac{g+d-1}{d-1} $.
  \begin{enumerate}
  \item If $M(E)$ is nonempty, then $E$ is a tame bundle.
  \item If $\lfloor E \rfloor \leq m$ then $M(E) \subset M(E[m])$.
  \item $M(E[m]) \subset M(E[m+1])$ for all $m$.   
  \item $M(E[m])$ is an irreducible subvariety of $\H_{d,g}^\dagger$ of codimension $g-(d-1)m + 1$ unless $m = \lfloor \frac{g+d-1}{d-1} \rfloor$, in which case $M(E[m]) =  \H_{d,g}^\dagger$.
  \item If $d$ is prime, then all the statements above hold with $M(-)$ replaced by $\td M(-)$ and $\mathcal H^\dagger_{d,g}$ replaced by $\td {\mathcal H}^\dagger_{d,g}$.
  \end{enumerate}
\end{theorem}
In the proof, we use a theorem of Coppens, which we state using our setup.
\begin{theorem}[\cite{coppens_existence}]\label{thm:coppens}
  For all $m$ satisfying $\frac{g+d-1}{{d \choose 2}} \leq m \leq \frac{g+d-1}{d-1}$, there is a genus $g$ and degree $d$ cover $C \to \P^1$ with Tschirnhausen bundle $E[m]$. Moreover, $C$ is birational onto its image under the associated scroll construction $C \to \F_m$.
\end{theorem}

\begin{proof}[Proof of \autoref{thm:maroniloci}]
We repeatedly use simultaneous normalization in the following way: Suppose we have a family $\mathcal C \to \Delta$ of reduced irreducible curves of geometric genus $g$. Then the normalization $\mathcal C^\nu$ of $\mathcal C$ gives a family $\mathcal C^\nu \to \Delta$ of smooth curves of genus $g$ \cite{teissier80:_resol_i}. For the ease of reading, we do not make this process explicit every time.

The first statement follows from \autoref{thm:constraintsE} and \autoref{thm:ohbuchi}. 
  
For the second statement, first note that if $[\overline C] \in \V_g(\F_m,d\tau)$ is a general point and $\nu \from C \to \overline C$ the normalization, then $C \to \P^1$ is simply branched and has Tschirnhausen bundle $E[m]$. Indeed, we can get a $[\overline C] \in \V_g(\F_m,d\tau)$ with Tschirnhausen module $E[m]$ by applying the associated scroll construction to a cover given by \autoref{thm:coppens}. By \autoref{thm:severi_simply_branched}, we may deform such $[\overline C]$ so that the normalization is simply branched. By the genericity of $E[m]$, the normalization of the deformed curve also has Tschirnhausen bundle $E[m]$. Now, suppose $\lfloor E \rfloor \leq m$ and $[C \to \P^1]$ is a point with Tschirnhausen bundle $E$. Then the associated scroll construction gives $\nu \from C \to \F_m$. Since $\alpha$ is simply branched, $\nu$ is birational onto its image. Then $\nu(C)$ is the limit of curves in $\V_g(\F_m,d\tau)$ whose normalization has Tschirnhausen bundle $E[m]$. The second statement follows.

The third statement is a corollary of the second statement.
 
For the fourth statement, suppose $m = \lfloor \frac{g+d-1}{d-1} \rfloor$. Then $E[m]$ is balanced, so $M(E[m]) = \H_{d,g}^\dagger$.  Suppose $m < \lfloor \frac{g+d-1}{d-1} \rfloor$. Let $U \subset \V_g(\F_m,d\tau)$ be the locus of nodal curves of geometric genus $g$ whose normalization is simply branched over $\P^1$. Then $U$ is a smooth open subset of $\V_g(\F_m,d\tau)$. Normalization of the universal family of curves in $\F_m$ of geometric genus $g$ gives a family of smooth curves of genus $g$ with a simply branched map of degree $d$ to $\P^1$ (induced from $\F_m \to \P^1$.) By definition, the image is in $M(E[m])$. We thus get a dominant map 
\[ q \from U \to M(E[m]).\]
The fiber of $q$ over $[\alpha \from C \to \P^1]$ corresponds to the global sections of $O_C(m)$ that project non-trivially onto $E^\vee(m)$. For general $\alpha \in M(E[m])$, we have $E_\alpha = E[m]$. Also, since $m < \lfloor \frac{g+d-1}{d-1} \rfloor$, the bundle $E[m]$ has a unique $O(m)$ summand and all other summands have degree greater than $m$. Therefore, the general fiber of $q$ has dimension $m+2$. From the dimension of $\V_g(\F_m,d\tau)$, we get
\[ \dim M(E[m]) = \dim \V_g(\F_m,d\tau) - (m+2) = (d-1)m + g + 2d - 3.\]
Since $\dim \H^\dagger_{d,g} = 2g+2d-2$, the fourth statement follows.
 
For the last statement, note that all the arguments hold for $\td M(E)$ if $d$ is prime, since the associated scroll construction $\nu \from C \to \F_{m}$ is automatically birational onto its image.
\end{proof}

\autoref{thm:maroniloci} gives us good control on the dimensions of the Maroni loci for $E$ based on the minimal summand of $E$. We must now consider those $E$ which are non-generic, but nonetheless have the same minimal summand as the generic Tschirnhausen bundle. Set $k = \lfloor \frac{g+d-1}{d-1} \rfloor$. Then 
\[E[k] = O(k)^{\oplus r} \oplus O(k+1)^{\oplus d-r-1},\] 
where $0 < r \leq d-1$. A general cover $\alpha \in \H^\dagger_{d,g}$ has $E[k]$ as its Tschirnhausen bundle. Let $E'$ be any tame bundle, and set $s := h^{0}({E'}^\vee(k))$.  Upper semicontinuity implies $s \geq r$. Suppose $s > r$. Define
\[M^\circ(E') =\left\{ \alpha \in \H^\dagger_{d,g} \mid E_{\alpha} \cong E' \right\}.\]
Then $M^\circ(E')$ is locally closed, and $\overline {M^\circ(E')} = M(E')$.
\begin{lemma}\label{maroni:codim}
  Under the assumptions above, let $Z \subset M^\circ(E')$ be any irreducible component. Then the codimension of $\overline Z$ in $\H^\dagger_{d,g}$ is at least $(s-r)+1$.
\end{lemma}
\begin{proof}
  Let $z = \dim Z$. We use the associated scroll construction over $Z$. We have an open subset $U$ of a vector bundle of rank $s+k+1$ over $Z$ and a morphism $U \to \V_g(\F_k, d\tau)$. Since $E' \neq E[k]$, the closure of the image of $U$ is a proper subvariety of $\V_g(\F_k,\tau)$. In particular, we have $\dim U < \dim \V_g(\F_k, d\tau) = dk+2d+g-1$. The lemma follows from this inequality.
\end{proof}

We now have the tools to determine all the Maroni divisors.
\begin{proposition}\label{maronidivisor} 
  The Maroni locus $M(E) \subset \H_{d,g}$ is a divisor if and only if $g=(k-1)(d-1)$ for some integer $k\geq1$, and $E= E[k-1] = O(k-1) \oplus O(k)^{\oplus d-3} \oplus O(k+1)$.  Furthermore, in this situation, $M(E[k-1])$ is irreducible.
\end{proposition}
\begin{proof}
  If $\lfloor E \rfloor = k = \lfloor \frac{g+d-1}{d-1} \rfloor$, then the statement follows by applying \autoref{maroni:codim}.  If, on the other hand, $\lfloor E \rfloor <   \lfloor \frac{g+d-1}{d-1} \rfloor$, then the statement follows from statement $4$ of \autoref{thm:maroniloci}.
\end{proof}

We record a particularly interesting case of the irreducibility of the Maroni divisor.
\begin{corollary}\label{giesekerpetri}
  Let $g=2(d-1)$. Then $M(E[2]) \subset \H_{d,g}$ is irreducible, and it is the ramification locus of the generically finite and dominant forgetful map $\mu \from \H_{d,g} \to \mathcal{M}_g$. 
\end{corollary}
\begin{proof}
The irreducibility statement follows from \autoref{thm:maroniloci}. To show that $M(E[2])$ is the ramification locus of $\mu$, consider $[\alpha \from C \to \P^1] \in \H_{d,g}$ and the map of sheaves: 
\[ 0 \to \alpha^{*}(T_{\P^1}) \to T_C \to N_{\alpha}\to 0.\]
The tangent space to $\H_{d,g}$ at $\alpha$ is $H^0(C, N_\alpha)/\alpha^*H^0(\P^1, T_{\P^1})$ and the tangent space to $\M_g$ at $C$ is $H^1(C, T_C)$. The map
 \[d\mu \from H^{0}(C, N_{\alpha})/\alpha^{*}H^{0}(\P^1, T_{\P^1}) \to H^{1}(C, T_C)\]
fails to be surjective precisely when $H^1(C, \alpha^*T_{\P^1}) \neq 0$, that is, when $\alpha \in M(E[2])$.
\end{proof}

\subsection{Linear independence of $T$, $D$, and $\Delta$}
In this section, we prove that the divisorial components of the boundary of $\td{\mathcal H}_{d,g}$ are linearly independent. Define the closed loci $T$, $D$, $\Delta$ in $\td{\mathcal H}_{d,g}$ by
\begin{align*}
  T &= \overline{\{ [\alpha \from C \to \P^1] \mid \alpha^{-1}(q) = 3p_1 + p_2 + \dots + p_{d-2}\text{ for some $q$ and distinct $p_i$.}\}}\\
  D &= \overline{\{ [\alpha \from C \to \P^1] \mid \alpha^{-1}(q) = 2p_1 + 2p_2 + p_3 + \dots + p_{d-2} \text{ for some $q$ and distinct $p_i$.}}\} \\
  \Delta &= \overline{\{ [\alpha \from C \to \P^1] \mid C \text{ is singular.}}\}
\end{align*}
These three loci correspond to the three possibilities of the limit when two branch points of a branched cover come together. Note that $T$, $D$, and $\Delta$ are irreducible and their union is the complement of $\td \H_{d,g}$ in $\H_{d,g}$.
\begin{proposition}\label{thm:lin_indep_boundary}
  For $d \geq 4$, the classes of $T$, $D$, and $\Delta$ are linearly independent in $\Pic_\Q(\td \H_{d,g})$. For $d \geq 3$, the same is true for the classes of $T$ and $\Delta$.
\end{proposition}
\begin{proof}
  We construct curves with non-singular intersection matrix with our divisors. For this, a slight enlargement of $\td \H_{d,g}$ is more convenient. Define $\td \H^{ns}_{d,g}$ as the moduli space of $[\alpha \from C \to \P^1]$ where $C$ is an at worst nodal curve of arithmetic genus $g$, not necessarily irreducible, but without any separating nodes, and $\alpha$ is a map of degree $d$. The target $\P^1$ is taken to be unframed. It is easy to see that $\td \H_{d,g}$ is a dense open subset of $\td \H_{d,g}^{ns}$ with codimension two complement. Abusing notation, we denote the closures of $T$, $D$, and $\Delta$ in $\td \H_{d,g}^{ns}$ by the same letters. It suffices to prove the proposition for $\td \H_{d,g}^{ns}$.

  We now construct test curves in $\td \H_{d,g}^{ns}$. Pick non-negative integers $g_1$ and $g_2$ with $g_1+g_2 = g-1$ and positive integers $d_1$ and $d_2$ with $d_1+d_2 = d$. Take a family $\alpha_b \from X_b \to \P^1$ of covers of degree $d_1$ and genus $g_1$, where $b$ denotes a parameter on a smooth complete curve $B$. Assume that we have two sections $p, q \from B \to X$ with $\alpha_b(p_b) = 0$ and $\alpha_b(q_b) = \infty$ for all $b \in B$.  Take $\beta \from E \to \P^1$ to be a fixed simply branched cover of degree $d_2$ and genus $g_2$, unramified over $0$ and $\infty$, and let $p', q' \in E$ be two points over $0$ and $\infty$ respectively. Our test curve in $\td \H_{d,g}^{ns}$ is given by the family $\gamma_b \from C_b \to \P^1$, where $C_b$ is obtained by gluing $(X_b, p_b, q_b)$ to the constant family $(E, p',q')$, and $\gamma_b \from C_b \to \P^1$ is induced from $\alpha \from X_b \to \P^1$ and $\beta \from E \to \P^1$. The construction is depicted in \autoref{fig:cover_attach}. 
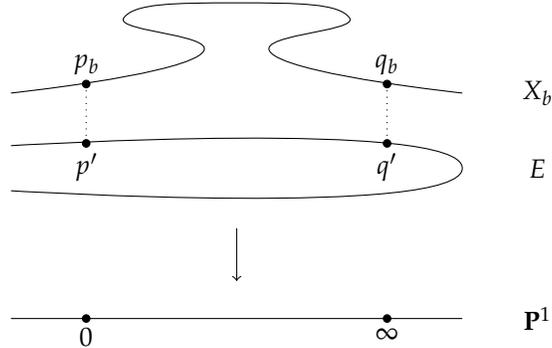
\begin{figure}[ht]
  \centering
  \begin{tikzpicture}
  \draw[smooth, tension=1] plot coordinates {(-3,0) (-.5, .5)
    (-1.5,1) (0,1.2) (1.5,1) (.5,0.5) (3,0)};

  \draw[fill] (-2,.12) circle (.05) node [above] {$p_b$};
  \draw[fill] (2,.12) circle (.05) node [above] {$q_b$};
  \draw (4,0) node {$X_b$};

  \begin{scope}[shift={(0,-1)}]
    \draw[smooth, tension=4] plot coordinates {(-3, 0.3) (3,0) (-3,-0.3)};
    \draw[fill] (-2, .33) circle (0.05) node [below] {$p'$};
    \draw[fill] (2, .33) circle (0.05) node [below] {$q'$};
  \draw (4,0) node {$E$};
  \end{scope}

  \draw[dotted] (-2,.12) edge (-2,-.67);
  \draw[dotted] (2,.12) edge (2,-.67);

  \begin{scope}[shift={(0,-3)}]
    \draw (-3,0) edge (3,0) (4,0) node {$\P^1$};
    \draw[fill] (2,0) circle (0.05) node[below] {$\infty$};
    \draw[fill] (-2,0) circle (0.05) node[below] {$0$};
  \end{scope}
  
  \path[->] (0,-1.8) edge (0,-2.5);
\end{tikzpicture}

\caption{We construct families of covers parametrized by $b \in B$ by
  attaching a variable family of covers $\alpha_b \from X_b \to \P^1$
  to a fixed cover $\beta \from E \to \P^1$.}
\label{fig:cover_attach}
\end{figure}

  Let $T_\alpha$, $D_\alpha$, and $\Delta_\alpha$ denote the pullbacks of the divisor classes $T$, $D$, and $\Delta$ along the map from $B$ to $\td \H_{d_1,g_1}$ given by $\alpha_b$. Define $T_\gamma$, $D_\gamma$, and $\Delta_\gamma$ likewise. Let $e$ be the intersection number of $\Br(\alpha)$ with a horizontal section of $\P^1 \times B$. Denote by $[p]$ (resp. $[q]$) the class of  $p(B)$ (resp. $q(B)$) on $X$.
  \begin{claim}
    With the notation above, we have
    \begin{align*}
      \deg T_\gamma &= \deg T_\alpha + 3([p]+[q]) \cdot \Ram(\alpha),\\
      \deg D_\gamma &= \deg D_\alpha + (2g_2+2d_2-2)e + 4e - 4([p]+[q])\cdot\Ram(\alpha), \text{ and }\\
      \deg \Delta_\gamma &= \deg \Delta_\alpha + [p]^2 + [q]^2.
    \end{align*}
  \end{claim}
  \begin{proof}[Proof of the claim]
    The pullback of the line bundle $O(\Delta)$ from $\td \H_{d,g}^{ns}$ to $B$ is given by
    \[(N_{p/X} \otimes N_{p'/E}) \otimes (N_{q/E} \otimes N_{q'/E}) \otimes O_B(\Delta_\alpha),\]
    where $N_{p/X}$ denotes the normal bundle of $p$ in $X$, and so on. The third equation follows.

    For a generic $b \in B$, the point of $\td \H_{d,g}^{ns}$ given by $\gamma_b \from C_b \to \P^1$ does not lie in $T$ or $D$. We have the following specializations:
    \begin{enumerate}
    \item $\alpha_b \from X_b \to \P^1$ has a fiber of the form $3p_1 + p_2 + \dots$. Such $b$'s are precisely the points of $T_\alpha$, each contributing $1$ to $\deg T_\gamma$. 
    \item $\alpha_b \from X_b \to \P^1$ has a fiber of the form $2p_1 + 2p_2 + p_3 + \dots$. Such $b$'s are precisely the points of $D_\alpha$, each contributing $1$ to $\deg D_\gamma$
    \item A branch point of $\alpha_b \from X_b \to \P^1$ coincides with a branch point of $\beta \from E \to \P^1$. There are $(2g_2+2d_2-2)e$ such $b$'s, each contributing $1$ to $\deg D_\gamma$.
    \item $p_b$ (resp. $q_b$) is a ramification point of $\alpha_b$. We compute the intersection multiplicity of $B$ with $T$ and $D$ at such a point by looking at a versal deformation space of $\gamma_b$. We may restrict $\gamma_b$ over an analytic neighborhood $U$ of $0$ (resp. $\infty$). Let $x$ be a coordinate on $U$. Then $\gamma_b^{-1}(U) \to U$ has the form
      \[U[y]/(y^3-xy) \sqcup U \sqcup \dots \sqcup U \to U.\]
      A versal deformation of this cover is given over $\spec \C[s,t]$ by
      \[U[y]/(y^3-xy-sx-t) \sqcup U \sqcup \dots \sqcup U \to U.\] In
      $\spec \C[s,t]$, the divisor $D$ does not contain the origin,
      and hence the intersection number of $B$ with $D$ at $b$ is
      $0$. The divisor $T \subset \spec \C[s,t]$ is defined by $t =
      0$. The curve $B$ approaches the origin along the locus where
      $U[y]/(y^3-xy-sx-t)$ is singular, namely along $s^3 + t = 0$. We
      deduce that the intersection number of $B$ with $T$ at $b$ is
      $3$. There are $[p] \cdot \Ram(\alpha)$ (resp. $[q] \cdot
      \Ram(\alpha)$) such $b$'s, each contributing 3 to $\deg
      T_\gamma$.

    \item $p_b$ (resp. $q_b$) is not a ramification point of
      $\alpha_b$, but lies over a branch point. Again, we look at a
      versal deformation of $\gamma_b$. In this case,
      $\gamma_b^{-1}(U) \to U$ has the form
        \[ U[y]/(y^2-x) \sqcup U[z]/(z^2-x^2) \sqcup U \sqcup \dots \sqcup U \to U.\]
      A versal deformation of this cover is given over $\spec \C[s,t]$ by
      \[ U[y]/(y^2-x) \sqcup U[z]/(z^2-x^2-sx-t) \sqcup U \sqcup \dots
      \sqcup U \to U.\] In $\spec \C[s,t]$, the divisor $T$ does not
      contain the origin, and hence the intersection number of $B$
      with $T$ at $b$ is $0$. The divisor $D \subset \spec \C[s,t]$ is
      defined by $t = 0$. The curve $B$ approaches the origin along
      the locus where $U[z]/(z^2-x^2-sx-t)$ is singular, namely along
      $s^2-4t = 0$. We deduce that the intersection number of $B$ with
      $D$ at $b$ is $2$. Let us count the number of such points, first
      for $p_b$, and analogously for $q_b$. The points $b$ for which
      $p_b$ is not a ramification point but lies over a branch point
      correspond to the intersection points of $\Br(\alpha) \cap \{0\}
      \times B$ which are not the images of the points of
      $\Ram(\alpha) \cap p(B)$. Note, however, that the image of a
      point of $\Ram(\alpha) \cap p(B)$ is actually a point of
      tangency of $\Br(\alpha)$ with $\{0\} \times B$, and hence
      contributes $2$ to the intersection number $e = \Br(\alpha)
      \cdot \{0\} \times B$. The remaining count, which we want, is
      therefore $e - 2[p]\cdot \Ram(\alpha)$. Similarly, the count for
      $q_b$ is $e - 2[q] \cdot \Ram(\alpha)$.
     \end{enumerate}
     The expressions for $T_\gamma$ and $D_\gamma$ follow from combining the above contributions.
  \end{proof}
  Returning to the proof of the proposition, consider the following three particular test curves for $d \geq 4$.
  \begin{enumerate}[$B_1$:]
  \item Take $\alpha_b \from X_b \to \P^1$ to be a family of hyperelliptic curves of genus $g-1$ obtained by taking a double cover $X \to \P^1 \times \P^1$ branched along a curve of type $(2g, 2)$. To have sections $p$ and $q$ of $X$ over $\{0\} \times \P^1$ and $\{\infty\} \times \P^1$, let the branch divisor be tangent to $\{0\} \times \P^1$ and $\{\infty\} \times \P^1$. Take $E$ to be a smooth rational curve and $\gamma \from E \to \P^1$ a generic cover of degree $d-2$.
  \item Take $\alpha_b \from X_b \to \P^1$ to be a family of trigonal curves of genus $g-1$ obtained by taking a general pencil on $\F_0$ in the linear system $|((g+1)/2,3)|$ if $g$ is odd, or on $\F_1$ in the linear system $|3\cdot\text{Directrix} + (g/2+2) \cdot \text{Fiber}|$ if $g$ is even. Two base-points give $p_b$ and $q_b$. Take $E$ to be a rational curve and $\gamma \from E \to \P^1$ a general cover of degree $d-3$.
  \item Take $\alpha_b \from X_b \to \P^1$ to be a family of hyperelliptic curves of genus $g-2$ as in $B_1$. Take $E$ to be a smooth genus $1$ curve and $\gamma \from E \to \P^1$ a generic cover of degree $d-2$. This curve exists only for $d \geq 4$. 
  \end{enumerate}
  Using the claim, we get the following non-singular intersection matrix.
  \[
  \begin{tabular}[h]{c| c c c}
    & $T$ & $D$ & $\Delta$\\
    \hline
    $B_1$&$6$ & $4d-12$ &$8g-6$\\
    $B_2$&$3g+9$& $8d-24$& $7g-3$\\
    $B_3$&$6$ &$4d-8$ & $8g-14$\\
  \end{tabular}
  \]
  For $d = 3$, we take a pencil in $\F_0$ or $\F_1$ as in $B_1$, but of trigonal curves of genus $g$, without any $E$. Then the middle column vanishes, and the second row becomes $(3g+6, 0, 7g+6)$, which is linearly independent from the first row.
\end{proof}

\section{Degree three}
Let $C$ be a curve of genus $g$ and $\alpha \from C \to \P^1$ a map of degree three. The relative canonical map embeds $C$ as a divisor in a $\P^1$-bundle $\P E$ over $\P^1$, where $E$ is a vector bundle of rank two and degree $g+2$.

Let
\[ E^\gen = O\left(\left\lfloor \frac{g+2}{2}\right\rfloor \right) \oplus O\left(\left\lceil \frac{g+2}{2} \right\rceil \right)\]
be the most generic vector bundle on $\P^1$ of rank $2$ and degree $g+2$. Set
\[ U_{E^\gen} := \{ \alpha \in \td \H_{3,g} \mid E_\alpha \cong E^\gen \}.\]
Note that $U_{E^\gen}$ is an open subset of $\td \H_{3,g}$.

\begin{proposition}\label{thm:maroni3}
  The complement of $U_{E^\gen}$ in $\td \H_{3,g}$ is a divisor if and only if $g$ is even, in which case it is irreducible.
\end{proposition}
\begin{proof}
  This is the degree 3 case of \autoref{maronidivisor}.
\end{proof}

Let $\pi \from \P E^\gen \to \P^1$ be the projection. Set
\[V = H^0(\P^1, \Sym^3 E^\gen \otimes {\det E^\gen}^\vee).\]
Elements of $\P_{\sub}V$ correspond to divisors in the linear series of the line bundle $O_{\P E^\gen}(3) \otimes \pi^* (\det E^\gen)^\vee$ on $\P E^\gen$. Let $C_v \subset \P E^\gen$ be the divisor corresponding to $v \in V$. Let $V^\circ \subset \P_\sub V$ be the open locus consisting of $v \in V^\circ$ for which $C_v$ is irreducible and at worst nodal. Let $G := \Aut(\pi)$ be the group of automorphisms of $\P E^\gen$ over $\P^1$. Then $G$ acts on $V^\circ$. The assignment 
\[v \mapsto [\pi \from C_v \to \P^1] \]
gives a map
\[q \from V^\circ \to \td \H_{3,g}^\dagger.\]
Denote by $U^\dagger_{E^\gen}$ the preimage of $U_{E^\gen}$ under $\td \H^\dagger_{3,g} \to \td \H_{3,g}$.
\begin{proposition}\label{thm:3_quotient}
  The image of $q$ is $U^\dagger_{E^\gen}$. The fibers of $q$ consist of single $G$-orbits.
\end{proposition}
\begin{proof}
  For brevity, set $E = E^\gen$. For $v \in V^\circ$, consider the
  sequence
  \[ 0 \to O_{\P E}(-3) \otimes \pi^* \det E \to O_{\P E} \to O_{C_v} \to 0.\]
  Applying $R\pi_*$, we get
  \begin{equation}\label{eqn:rel_can_U}
    0 \to O_{\P^1} \to \pi_* O_{C_u} \to  E^\vee \to 0,
  \end{equation}
  which says that the Tschirnhausen bundle of $C_u \to \P^1$ is $E$. Conversely, from the Casnati--Ekedahl resolution, it follows that every point of $U^\dagger_{E^\gen}$ is in the image of $q$.

  Let $u, v \in U^\dagger_{E^\gen}$ be in a fiber of $q$. Then there is an isomorphism $C_u \to C_v$ over the identity of $\P^1$. The sequence \eqref{eqn:rel_can_U} for $C_u$ and $C_v$ shows that such an isomorphism induces an isomorphism $E \to E$. The induced automorphism of $\P E$ over $\P^1$ takes $C_u$ to $C_v$ and hence $u$ to $v$. 
\end{proof}
\begin{proposition}\label{thm:PRC3}
  [Picard rank conjecture for degree three]
  We have $\Pic_{\Q}\H_{3,g} = 0$.
\end{proposition}
\begin{proof}
  Retain the notation introduced above. For brevity, set $U = U_{E^\gen}$ and $U^\dagger = U_{E^\gen}^\dagger$. Then $V^\circ \to U^\dagger$ is a quotient by $G$ and $U^\dagger \to U$ is a quotient by $\PGL_2$. By \autoref{thm:pic_quotient} and \autoref{thm:3_quotient}, we have 
  \begin{align*}
    \rk \Pic_\Q U \leq \rk \Pic_\Q U^\dagger + \rk \chi (\PGL_2) &= \rk \Pic_\Q U^\dagger \\
    &\leq \rk \Pic_\Q V^\circ + \rk \chi(G) \leq 1 + \rk \chi (G).
  \end{align*}
  The final inequality follows because $V^\circ$ is an open subset of a projective space. Let $e$ be the number of divisorial components of $\td \H_{3,g} \setminus U$. We then get the bound
  \begin{align*}
    \rk \Pic_\Q \td \H_{3,g} \leq \rk \Pic_\Q U +  e \leq 1 + \rk\chi(G) + e.
  \end{align*}
  If $g$ is even, then
  \begin{align*}
    G &= \PGL_2 \\
    \rk\chi(G) &= 0 \\
    e &= 1 \quad \text{by \autoref{thm:maroni3}}.
  \end{align*}
  If $g$ is odd, then
  \begin{align*}
    G &= \left\{\begin{pmatrix}
        a & l \\
        & b
      \end{pmatrix} \mid a, b \in \C^*, l \in H^0(\P^1, O(1)) \right\} \big{/} \C^* \\
    \rk\chi(G) &= 1 \\
    e &= 0 \quad \text{by \autoref{thm:maroni3}}.
  \end{align*}
  In either case, we have
  \[ \rk \Pic_\Q \td \H_{3,g} \leq 2.\]
  By \autoref{thm:lin_indep_boundary}, the classes in $\Pic_\Q(\td \H_{3,g})$ of the two components of $\td \H_{3,g} \setminus \H_{3,g}$ are linearly independent. Therefore, we get $\Pic_\Q\H_{3,g} = 0$ as desired.
\end{proof}

\section{Degree four}
Let $C$ be a curve of genus $g$ and $\alpha \from C \to \P^1$ a map of
degree four. The relative canonical map embeds $C$ into a
$\P^2$-bundle $\P E$ over $\P^1$, where $E$ is a vector bundle of rank
three and degree $g+3$. The Casnati-Ekedahl structure theorem provides
the following resolution of $O_C$:
\[ 0 \to \pi^{*} \det E (-4) \to \pi^{*}F(-2) \to O_{\P E} \to O_C \to 0,\]
where $F$ is a vector bundle of rank two and degree $g+3$.

Explicitly, we can describe $C \subset \P E$ as follows. Write $F =
O(a) \oplus O(b)$, where $a+b = g+3$, and $a \leq b$. Let $h$ denote
the divisor class associated to $O_{\P E}(1)$ on $\P E$ and $f$ the
class of the fiber of the projection $\pi \from \P E \to \P^1$. Then
the curve $C$ is the complete intersection of two divisors
\[C = Q_{a} \cap Q_{b},\] 
where $[Q_{a}] = 2h-af$ and $[Q_{b}] = 2h-bf$.

Even more explicitly, we can describe the equations of $Q_a$ and $Q_b$
as follows.  Write $E = O(m_{1}) \oplus O(m_{2}) \oplus
O(m_{3})$. Over an open set $U \subset \P^1$, let $X, Y,$ and $Z$
denote the relative coordinates on $\P E|_U$ corresponding to the
three summands of $E$. Assume that $m_{1} \leq m_{2} \leq m_{3}$. Over
$U$, the divisor $Q_{a}$ is the zero locus of a form
\begin{equation}\label{relativequadrica}
  p_{1,1}X^2 + p_{1,2}XY + p_{1,3}XZ + p_{2,2}Y^2 + p_{2,3}YZ + p_{3,3}Z^2
\end{equation}
where $p_{i,j}$ is the restriction to $U$ of a global section of
$O(m_{i} + m_{j} -a)$. Similarly, over $U$, the divisor $Q_{b}$ is the
zero locus of a form
\begin{equation}\label{relativequadricb}
q_{1,1}X^2 + q_{1,2}XY + q_{1,3}XZ + q_{2,2}Y^2 + q_{2,3}YZ + q_{3,3}Z^2
\end{equation}
where $q_{i,j}$ is the restriction to $U$ of a global section of
$O(m_{i} + m_{j} -b)$.

The irreducibility of $C$ puts some restrictions on the possible
$(E,F)$. Indeed, if $p_{1,1} = q_{1,1} = 0$, then the section $[X:Y:Z]
= [1:0:0]$ of $\P E$ is contained in both $Q_{a}$ and $Q_{b}$, making
$C = Q_a \cap Q_b$ reducible. An irreducible $C$ thus forces
\begin{equation}\label{splitdirectrix}
  2m_1 \geq a.
\end{equation}

\begin{proposition}\label{thm:Mef_irred}
  Let $E$ be a vector bundle of rank 3 and degree $g+3$ and $F$ a
  vector bundle of rank 2 and degree $g+3$. If the locus $M(E, F)$
  is non-empty, then it is irreducible and unirational.
\end{proposition}
\begin{proof}
  Consider the dense open subset $M^\circ(E, F) \subset M(E,F)$ corresponding to $\alpha \in \H_{4,g}$ that have
  $E_\alpha \cong E$ and $F_\alpha \cong F$. It suffices to prove the
  statement for $M^\circ(E,F)$.

  Consider the vector space 
  \[ V := H^0(\P^1, F^\vee \otimes \Sym^2 E).\] Elements of $V$
  correspond to maps $\pi^* F(-2) \to O_{\P E}$. Let $V^\circ \subset
  V$ be the open subset where the ideal generated by the image of
  $\pi^*F(-2)$ defines a smooth curve, simply branched over
  $\P^1$. Then $V^\circ$ surjects onto $M^\circ(E,F)$.
\end{proof}

\begin{remark}
  From the dominant map $V^\circ \to M(E,F)$ in the proof of
  \autoref{thm:Mef_irred}, it is easy to compute the codimension of
  $M(E,F)$ in $\H_{4,g}$, which is
  \[ \codim M(E, F) = \dim \Ext^1(E,E) + \dim \Ext^1(F,F) - \dim
  \Ext^1(F, \Sym^2 F).\] We may think of $\dim \Ext^1(E,E) + \dim
  \Ext^1(F,F)$ as the `expected codimension.' The next example shows
  that the actual codimension is not always the expected codimension.
\end{remark}

\begin{example} \label{ex:1}
  Let $E = O(m) \oplus O(2m) \oplus O(g+3-3m)$, where $\lceil
  \frac{g+3}{6} \rceil \leq m < \frac{g+3}{5}$. To get an irreducible
  curve $C$, the only possibility for $F$ is $F = O(2m) \oplus
  O(g+3-2m)$, by \eqref{splitdirectrix}.  The resulting locus
  $M(E,F)$ is not of expected codimension because $\dim \Ext^{1}
  (F,\Sym^2E)$ is nonzero.
\end{example}

\begin{example} \label{ex:2}
  The Maroni locus $M(E)$ may be reducible. Let $g = 12$, and
  consider the bundle $E = O(3) \oplus O(5) \oplus O(7)$.  Then the
  reader can easily check (using Bertini's theorem) that $M(E,F)$ and
  $M(E,F')$ are nonempty and of \emph{equal} codimension $\dim
  \Ext^{1}(E,E)$ for the bundles $F = O(6) \oplus O(9)$ and $F' = O(5)
  \oplus O(10)$.  Therefore $M(E,F)$ and $M(E,F')$ are two components
  of $M(E)$. It is easy to see by analyzing the explicit equations
  that these are the only components of $M(E)$.
\end{example}

Let $E^\gen$ (resp. $F^\gen$) be the most generic vector bundle on $\P^1$ of rank 3 (resp. 2) and degree $g+3$. Define
\begin{align*}
  U_{E^\gen} &:= \{ \alpha \in \td \H_{4,g} \mid {E}_\alpha \cong {E^\gen}\}, \\
  U_{F^\gen} &:= \{ \alpha \in \td \H_{4,g} \mid {F}_\alpha \cong {F^\gen}\}, \\
  U_{{E^\gen},{F^\gen}} &:= U_{E^\gen} \cap U_{F^\gen}.
\end{align*}
It is easy to see that these are are open subsets of $\td
\H_{d,g}$. Our next task is to identify the divisorial components of
their complements.

\begin{proposition}\label{thm:4_compUe}
  The subvariety $M := \td\H_{4,g} \setminus U_{E^\gen}$ is a
  divisor if and only if $g$ is divisible by three, in which case it
  is irreducible.
\end{proposition}
\begin{proof}
  This is the degree $4$ case of \autoref{maronidivisor}.
\end{proof} 

For the complement of $U_{F^\gen}$, we could do a careful analysis of the
defining equations of $C$ in $\P E$, as we will have to do for the
next case of $d = 5$. But we can take a more geometric approach using
the \emph{resolvent cubic construction}. Originally due to Recillas
\cite{recillas:73}, the construction can be described as
follows. For simplicity, we give an informal description, restricting
to simply branched covers. See \cite{casnati:trigconst} for a detailed
account. Consider a point $[\alpha \from C \to \P^1]$ of
$\H_{4,g}$. The resolution of $O_C$ as an $O_{\P E_\alpha}$ module
shows that $C \subset \P E_\alpha$ is the complete intersection of two
relative quadrics. A fiber of $\P F_\alpha \to \P^1$ naturally
corresponds to the pencil of conics in the corresponding fiber of $\P
E_\alpha \to \P^1$ containing the corresponding fiber of $C \to
\P^1$. Each such pencil contains three singular conics, counted with
multiplicity. The total locus of these singular conics forms a
trigonal curve $R(C) \subset \P F_\alpha$. Let $R(\alpha) \from R(C)
\to \P^1$ be the projection. We call $R(\alpha)$ the \emph{resolvent
  cubic} of $\alpha$. Using that $C \to \P^1$ is simply branched, it
is easy to check that $R(C)$ is smooth and the branch divisor of
$R(\alpha)$ coincides with the branch divisor of $\alpha$. In
particular, $R(C)$ has genus $g+1$. The association $\alpha \to
R(\alpha)$ defines a map
\[ R \from \H_{4,g} \to \H_{3,g+1},\] which we call the
\emph{resolvent cubic map}. The fiber of $R$ over a point $[D \to
\P^1] \in \H_{3,g+1}$ corresponds bijectively to the set of \'etale
double covers $D' \to D$ (see \cite{recillas:73} or
\cite[Theorem~6.5]{casnati:trigconst}). In particular, $R$ is a finite
morphism.

\begin{proposition}\label{thm:4_C(F)}
  Let $F$ be a vector bundle of rank 2 and degree $g+3$ on $\P^1$. The
  Casnati-Ekedahl locus $C(F) \subset \H_{4,g}$ is
  non-empty if and only if $\lfloor F \rfloor \geq \lceil
  \frac{g+3}{3} \rceil$. In this case, it is of the expected
  codimension $\dim \Ext^1(F,F)$.
\end{proposition}
\begin{proof}
  Consider a point $[\alpha \from C \to \P^1]$ of $\H_{4,g}$ and
  its resolvent cubic $R(\alpha) \from R(C) \to \P^1$. Since $R(C)
  \subset \P F_\alpha$, and $F_\alpha$ is a vector bundle of rank two
  and degree $(g+1)+2$, it must be the Tschirnhausen bundle of
  $R(C)$. That is, we have $E_{R(\alpha)} = F_\alpha$. By
  \cite{recillas:73}, the map $R$ is finite, and hence
  ${C}(F) = R^{-1}({M}(F)).$ Both of the statements about
  ${C}(F)$ now follow from the corresponding statements about $M(F)$.
\end{proof}

 \begin{proposition}\label{thm:4_compUf}
   Let $g \geq 4$. The subvariety $CE := \H_{4,g} \setminus
   U_{F^\gen}$ has codimension at least two if $g$ is even and is an
   irreducible divisor if $g$ is odd.
 \end{proposition}
\begin{proof}
  The image $R(U_{F^\gen}) \subset \H_{3,g+1}$ is the open locus of
  trigonal covers having $F^\gen$ as their Tschirnhausen bundle. The
  complement $Z := \H_{3,g+1} \setminus R(U_{F^\gen})$ has codimension
  at least two if $g+1$ is odd and it is the Maroni divisor if $g+1$
  is even (\autoref{thm:maroni3}). The complement $\H_{4,g} \setminus
  U_{F^\gen}$ is the preimage $R^{-1}(Z)$. Therefore, the statements
  about the codimension follow from the finiteness of $R$.

  For the question of reducibility, let $F = O(k-1) \oplus O(k+1)$
  with $k = (g+3)/2 \geq 3$. The claim is that $C(F)$ is irreducible
  when $g > 3$, and has two components when $g=3$. We have
  \[C(F) = \bigcup_{E} M(E,F).\]
  By \autoref{thm:Mef_irred}, the varieties $M(E,F)$ are
  irreducible. Therefore, every component of $C(F)$ must be of
  the form $M(E,F)$ for some $E$.

  Let $g > 3$ and suppose $E \neq E^\gen$. The inclusion $M(E, F)
  \subset M(E)$ and \autoref{maronidivisor} imply that $ M(E,
  F)$ is a divisor if and only if $M(E, F) =  M(E)$ and $E =
  O(m-1) \oplus O(m) \oplus O(m+1)$. By choosing two generic quadrics
  as in \eqref{relativequadrica} and \eqref{relativequadricb}, we can
  explicitly construct a curve in $ M(E, F^\gen)$, showing that
  $ M(E, F) \neq M(E)$. Thus, it follows that the only
  component of $C(F)$ is $M(E^\gen, F)$.
\end{proof}

\begin{example}
  The divisor $\H_{4,g} \setminus U_{F^\gen}$ is not irreducible
  for $g = 3$. Indeed, take $F = O(2) \oplus O(4)$. Then $M(E^\gen, F)$ is an irreducible component. Now consider the only
  other possibility for $E$, namely $E = O(1) \oplus O(2) \oplus
  O(3)$. By \eqref{splitdirectrix}, a cover in $M(E)$ must have
  $F = O(2) \oplus O(4)$. Furthermore, for this $E$ and $F$, we can
  choose the two quadrics generically and see that $M(E, F)$ is
  nonempty. Therefore, $M(E) = M(E, F)$ is another component
  of $\H_{4,g} \setminus U_{F^\gen}$.
\end{example}

Our next goal is to exhibit $U_{E^\gen, F^\gen}$ as a quotient. Let
$\pi \from \P E^\gen \to \P^1$ be the projection. For brevity, set $E
= E^\gen$ and $F = F^\gen$.  Set
\[ V := H^0(\P^1, {F}^\vee \otimes \Sym^2 E).\] An element $v \in
\P_{\sub}V$ corresponds to a map $\pi^* F(-2) \to O_{\P E}$. Let $C_v$
be the zero locus of the image of this map. Let $V^\circ \subset
\P_{\sub} V$ be the open locus consisting of $v \in \P_{\sub}V$ for
which $C_v$ is irreducible and at worst nodal. Let $G_F := \Aut(\P
F/\P^1)$ and $G_E := \Aut(\P E / \P^1)$. Then $G_F \times G_E$ acts on
$V^\circ$. The assignment
\[ v \mapsto [\pi \from C_v \to \P^1]\]
defines a map
\[q \from V^\circ \to \td \H^\dagger_{4,g}.\]
Denote by $U^\dagger_{E,F}$ the preimage of $U_{E,F}$ under $\td \H^\dagger_{4,g} \to \td \H_{4,g}$.
 \begin{proposition}\label{thm:4_quotient}
   The image of $q$ is $U^\dagger_{E^\gen,F^\gen}$. The fibers of $q$ consist of single $G$-orbits.
\end{proposition}
\begin{proof}
  The proof is exactly analogous to the proof of \autoref{thm:3_quotient}.
\end{proof}

\begin{proposition}\label{thm:PRC4}
  [Picard rank conjecture for degree four]
  We have $\Pic_{\Q}\H_{4,g} = 0.$
\end{proposition}
\begin{proof}
  Retain the notation introduced above. For brevity, set $U = U_{E^\gen, F^\gen}$ and $U^\dagger = U^\dagger_{E^\gen, F^\gen}$. By \autoref{thm:pic_quotient} and \autoref{thm:4_quotient}, we have
  \begin{align*}
    \rk \Pic_\Q U \leq \rk \Pic_\Q U^\dagger + \rk\chi(\PGL_2) &= \rk \Pic_\Q U^\dagger\\
    &\leq \rk \Pic_\Q V^\circ + \rk\chi(G) \leq 1 + \rk\chi(G).
  \end{align*}
  The final inequality follows because $V^\circ$ is an open subset of
  a projective space. Let $e$ be the number of divisorial components
  of $\td \H_{3,g} \setminus U$. We then get the bound
  \[ \rk \Pic_\Q \td \H_{4,g} \leq \rk \Pic_\Q U + e \leq 1 + \rk\chi(G)
  + e.\] Recall that $G = G_{F^\gen} \times G_{E^\gen}$.

  If $g$ is an odd multiple of $3$, then
  \begin{align*}
    G &= \PGL_{2} \times \PGL_{3} \\
    \rk\chi(G) &= 0\\
    e &= 2 \quad \text{corresponding to $M$ in \autoref{thm:4_compUe} and $CE$ in \autoref{thm:4_compUf}}.
  \end{align*}

  If $g$ is odd, but not divisible by $3$, then
  \begin{align*}
    G &= \PGL_2 \times G_E \\
    G_E &= \left\{\begin{pmatrix}
  a & b & l_{1}\\
  c & d & l_{2}\\
  0 & 0 & e
\end{pmatrix} \mid a, b, c, d, e \in \C, e(ad-bc) \in \C^{*},  l_{i} \in H^0(\P^1, O(1)) \right\} \big{/} \C^*. \\
\rk\chi(G) &= 1 \\
e &= 1 \quad \text{corresponding to $CE$ in \autoref{thm:4_compUf}}.
  \end{align*}

  If $g$ is even and divisible by $3$, then
  \begin{align*}
    G &= G_F \times \PGL_2 \\
    G_F &= \left\{\begin{pmatrix}
        a & l \\
        & b
      \end{pmatrix} \mid a, b \in \C^*, l \in H^0(\P^1, O(1)) \right\} \big{/} \C^*\\
    \rk\chi(G) &= 1\\
    e &= 1 \quad \text{corresponding to $M$ in \autoref{thm:4_compUe}}.
  \end{align*}
  
  If $g$ is even and not divisible by $3$, then
  \begin{align*}
    G &= G_F \times G_E \text{ where $G_F$ and $G_E$ are as in the previous two cases,}\\
    \rk\chi(G) &= 2 \\
    e &= 0.
  \end{align*}
  In all cases, we get
  \[ \rk \Pic_\Q\td \H_{4,g} \leq 3.\]
  By \autoref{thm:lin_indep_boundary}, the classes in $\Pic_\Q \td \H_{4,g}$ of the three components of $\td \H_{4,g} \setminus \H_{4,g}$ are linearly independent. Therefore, we get $\Pic_\Q\H_{4,g} = 0$ as desired.
\end{proof}

\section{Degree five}
Let $C$ be a curve of genus $g$ and $\alpha \from C \to \P^1$ a map of degree five. The relative canonical map embeds $C$ into a $\P^3$ bundle $\P E$ over $\P^1$, where $E$ is a vector bundle of rank four and degree $g+4$. The Casnati-Ekedahl structure theorem provides the following  resolution of $O_C$:
\[ 0 \to \pi^{*} \det E (-5) \to \pi^{*} ( {F}^\vee (\det E))(-3) \to \pi^{*}F(-2) \to O_{\P E} \to O_C \to 0\] 
where $F$ is a vector bundle of rank three and degree $2g+8$. 

Explicitly, we can describe $C \subset \P E$ as follows. The resolution is determined completely by the middle map
\[w \from \pi^{*} ( {F}^\vee (\det E))(-3) \to \pi^{*}F(-2).\] This
map may be viewed an element of the vector space $H^0(\P^1, F \otimes
F \otimes E (- \det E) )$. Due to a theorem of Casnati
\cite{casnati:five}, $w$ can be taken to be anti-symmetric, that is, in the
subspace
 \[ V := H^0(\P^1, \wedge^2F  \otimes E \otimes \det E^\vee ). \]
 
 Even more explicitly, we can describe the defining equations of $C$ as follows. Let 
 \begin{align*}
   F &= O(n_{1}) \oplus \dots \oplus O(n_{5}), \text{ where } n_{1} \leq \dots \leq n_{5}, \text{ and}\\
   E &= O(m_{1}) \oplus \dots \oplus O(m_{4}), \text{ where } m_{1} \leq \dots \leq m_{4}.   
 \end{align*}
 We represent an element $w \in V$ by a skew symmetric matrix of forms 
\begin{equation}\label{eqn:Matrix5}
  M_{w} = \begin{pmatrix}
0 & L_{1,2} & L_{1,3} & L_{1,4} & L_{1,5} \\
-L_{1,2} & 0 & L_{2,3} & L_{2,4} & L_{2,5} \\
-L_{1,3} & -L_{2,3} & 0 & L_{3,4} & L_{3,5}\\
-L_{1,4} & -L_{2,4} & - L_{3,4} & 0 & L_{4,5}\\
-L_{1,5}  & -L_{2,5} & -L_{3,5} & -  L_{4,5}& 0  \end{pmatrix}
\end{equation}
where $L_{i,j} \in H^{0}({\bf P}^1,E\otimes {\det E}^\vee \otimes O(n_i + n_j))$. In $\P E$, the curve $C_{w}$ is cut out by the $4 \times 4$ sub-Pfaffians of the matrix $M_{w}$.

The irreducibility of $C$ puts some restrictions on the possible
matrices. Indeed, suppose
\[L_{1,2} = L_{1,3} = 0.\] Then the Pfaffian $Q_5$ of the submatrix
obtained by eliminating the fifth row and column is \[Q_5 =
L_{1,2}L_{3,4} - L_{1,3}L_{2,4}+L_{2,3}L_{1,4} = L_{2,3}L_{1,4}.\]
Since $Q_5$ is reducible, it forces $C_w$ to be reducible.

Suppose further that $E = O(k)^{r} \oplus
O(k+1)^{4-r}$, where $ 0 \leq r \leq 3$.  Then the observation above
implies that the maximum of the degrees of the summands of $E\otimes
( {\det E}^\vee) \otimes O(n_1+ n_3)$ must be nonnegative, meaning
\[ n_1 +n_3 + k - (g+4) \geq -1.\] Since the $n_{i}$ are increasing,
we get the inequalities
\begin{equation}\label{firstentries}
  n_i+n_j + (k+1) - (g+4) \geq 0 \text{ for every $(i,j)$ with $i \neq j$ except $(i,j) = (1,2)$.}
\end{equation}

Let $E^\gen$ (resp. $F^\gen$) be the most generic vector bundle on $\P^1$ of rank 4 (resp. 5) and degree $g+4$ (resp. $2g+8$). Define $U_{E^\gen}$, $U_{F^\gen}$, and $U_{E^\gen, F^\gen}$ as before. These are the open subsets of $\td \H_{5,g}$ consisting of covers $\alpha$ for which $E_\alpha$, $F_\alpha$, and both $E_\alpha$ and $F_\alpha$ are the most generic.
 
\begin{proposition}\label{thm:5_compUe}
  The subvariety $M := \td \H_{5,g} \setminus U_{E^\gen}$ has codimension
  at least two if $g$ is not divisible by $4$, and has a unique
  divisorial component if $g$ is divisible by $4$.
\end{proposition}
\begin{proof}
  This is the degree $5$ case of \autoref{maronidivisor}.
\end{proof}

For the complement of $U_{F^\gen}$, we must analyze the defining equations of $C$ in $\P E$.
\begin{proposition}\label{thm:5_compUf}
  The subvariety $CE := \H_{5,g} \setminus U_{F^\gen}$ has codimension
  at least two if $g+4$ is not a multiple of $5$ (with the exception
  of $g = 3$ in which case the complement parametrizes hyperelliptic
  curves), and contains a unique divisorial component if $g+4$ is a
  multiple of $5$.
 \end{proposition}
 \begin{proof}
   We must characterize the Casnati-Ekedahl loci $C(F)$ which are divisorial. We have
   \[C(F) = \bigcup_{E}M(E,F).\] The loci $M(E, F)$ are irreducible by
   virtually the same argument as in \autoref{thm:Mef_irred} (In the
   proof, just take $V = H^0(\P^1, \wedge^2 F \otimes E \otimes \det
   E^\vee)$.) Therefore, any component of $C(F)$ must be of the form
   $M(E,F)$. From the explicit description of degree $5$ covers above,
   it is straightforward to compute that
   \[\codim M(E,F) = \dim \Ext^{1}(E,E) + \dim \Ext^{1}(F,F) -
   h^{1}(\wedge^2F \otimes E \otimes \det E^\vee ).\]

   Suppose $E \neq E^\gen$. Then $M(E, F) \subset M(E)$. By
   \autoref{maronidivisor}, $M(E)$ has codimension at least two unless
   $E = O(k) \oplus O(k+1)^{\oplus d-3} \oplus O(k+2)$. In this case,
   using the explicit description of degree $5$ covers, it is easy to
   construct covers $\alpha$ with $E_\alpha = E$ and $F_\alpha =
   F^\gen$. Thus, $M(E, F) \neq M(E)$, and since $M(E)$ is
   irreducible, $M(E,F) \subset M(E)$ has codimension at least
   one. Therefore, $M(E, F) \subset \H_{4,g}$ has codimension at least
   two.

   Therefore, for $M(E, F)$ to be divisorial, we must have $E = E^\gen$. In
   this case, we have
   \[ \codim M(E,F) = \dim \Ext^1(F,F) - h^1(\wedge^2 F \otimes E \otimes det E^\vee).\]

   Suppose $h^1(\wedge^2 F \otimes E\otimes\det E^\vee) = 0$. Note that $\dim \Ext^1(F,F) = 1$ if and only if 
   \[F = O(n-1) \oplus O(n) \oplus O(n) \oplus O(n) \oplus O(n+1).\]
   In this case $5n = 2(g+4)$, and hence $5$ divides $g+4$. 

   We are thus reduced to showing that $M(E, F)$ is not a divisor when $E = E^\gen$ and  
   \[h^1(\wedge^2 F \otimes E(-\det E)) > 0,\]
   with the exception of $g = 3$. Write 
   \[E = O(k)^{\oplus r} \oplus O(k+1)^{\oplus 4-r} \text{ where } 0 \leq r \leq 3,\]
   and 
   \[F = O(n_1) \oplus O(n_2) \oplus O(n_3) \oplus O(n_4) \oplus
   O(n_5), \text{ where } n_1 \leq \dots \leq n_5.\]

   Consider an anti-symmetric matrix
   \[
   M_{w} = (L_{i,j}) \quad 1 \leq i, j \leq 5,\] as in
   \eqref{eqn:Matrix5}, representing an element of $H^{0}(\wedge^2 F
   \otimes E \otimes \det E^\vee)$. Inequality
   \eqref{firstentries} implies that any contribution to
   $h^1(\wedge^2 F \otimes E \otimes \det E^\vee)$ must come from the
   $L_{1,2}$ entry. In other words, we have
   \[h^1(\wedge^2 F \otimes E(-\det E)) = h^{1}(E\otimes  {\det E}^\vee \otimes O(n_1 + n_2)).\]
   Since $E = E^\gen$, we have $h^{1}(E\otimes {\det E}^\vee \otimes O(n_1 + n_2)) > 0$ if and only if 
   \[n_1 + n_2 + (k+1)-(g+4) < 0.\]
   Hence, we get
   \begin{align*}
     h^{1}(E\otimes  {\det E}^\vee \otimes O(n_1 + n_2)) &= 4(-(n_1 + n_2 + k -(g+4))-1) - (4-r) \\
     &= 4g - 4(n_1+n_2+k) + r + 8.
   \end{align*}
   Equation~\eqref{firstentries} tells us that $n_1 + n_3 + (k+1) -(g+4) \geq 0$, which implies $n_2 < n_3$. Therefore, 
   \[\dim \Ext^{1}(F,F) \geq (2n_5 + 2n_4 + 2n_3) - 3(n_1 + n_2) - 6.\] 
   Combining the two, we get
   \[\dim \Ext^{1}(F,F) - h^{1}(E\otimes  {\det E}^\vee \otimes O(n_1 + n_2)) \geq 2n_5+2n_4 +2n_3 + n_1 + n_2 -3(g+4)-2.\]
   Using $n_1 + \dots + n_5 = 2(g+4)$,  the above inequality becomes
   \[\dim \Ext^{1}(F,F) - h^{1}(E\otimes {\det E}^\vee \otimes O(n_1 + n_2)) \geq (g+4) - (n_1 + n_2) - 2.\]
   Finally, by using the assumption $n_1 + n_2 + (k+1) - (g+4) < 0$, we conclude that 
   \[\codim M(E^\gen, F) = \dim \Ext^{1}(F,F) - h^{1}(E\otimes {\det E}^\vee \otimes O(n_1 + n_2)) > k-1.\]
   If $k > 1$, then we get $\codim M(E^\gen, F) > 1$ as desired. We
   consider the cases where $k = 1$ on an individual basis. These
   cases correspond to $0 \leq g \leq 4$.

   \begin{asparaitem}[\textbf{Case:}]
   \item $g = 4$. Then $E^\gen = O(2)^{\oplus 4}$ and $F^\gen = O(3)^{\oplus 4} \oplus O(4)$. The relative canonical map embeds $C$ in $\P E^\gen \simeq \P^3 \times \P^1$. The projection to $\P^3$ restricts to the canonical map on $C$. Therefore, if $C$ is non-hyperelliptic, then there is only one quadric in $\P^3$ containing the canonical model of $C$.  This means that the bundle $F$ has exactly one $O(4)$ summand, and hence $F \cong F^\gen$. The locus where $C$ is hyperelliptic  is easily seen to be codimension $2$ in $\H_{5,4}$. This exhausts all possibilities in this case.

   \item $g = 3$. Then $E^\gen = O(1) \oplus O(2)^{\oplus 3}$ and $F^\gen = O(2) \oplus O(3)^{\oplus 4}$. Consider the special bundle $F = O(2) \oplus O(2) \oplus O(3)^{\oplus 2} \oplus O(4)$.  Then \[\dim \Ext^1(F,F) - h^1(\wedge^2 F \otimes E \otimes \det E^\vee) = 1.\]
 Now consider a general $[\alpha \from C \to \P^1] \in M(E,F) \subset \H_{5,3}$.  Let $[X : Y :Z :W]$ denote the homogeneous coordinates (locally over $\P^1$) on $\P E$ corresponding to the summands of $E$. As usual, denote by $h$ the class of $O_{\P E}(1)$ and by $f$ the class of the fiber of $\P E \to \P^1$. Since $O(4)$ is a summand of $F$, there exists a unique effective divisor $Q$ of class $2h - 4f$ on $\P E$ which contains $C$. The quadric $Q$ may be written as the zero locus of a form 
 \[c_{0}Y^2 + c_1YZ + \dots +c_5W^2,\] where $c_{i}$ are constants. Let $p \from \P E \dasharrow \P^{2} \times \P^1$  be the projection from the section $[1: 0 : 0 : 0]$, and $g \from \P E \dasharrow \P^{2} \times \P^1 \to \P^{2}$ the composition with the projection onto the first factor.  Then the rational map $g$ is given by the linear system $|h-2f|$ on $\P E$, which restricts to the canonical series on $C$.  However, the fact that $C$ lies on the relative quadric $Q$ means that the image $g(C)$ is exactly the conic defined by the equation for $Q$. Thus, $C$ is hyperelliptic.  
 Given the above geometric understanding of the $O(4)$ summand of $F$, it is easy to show that if we begin with a hyperelliptic curve $C$, and a degree $5$ map $\alpha \from C \to  \P^{1}$, then $F_{\alpha}$ must contain a unique $O(4)$ summand. By the inequalities in \eqref{firstentries}, there are no other choices for $F$.
 \item $g=1,2$. In these cases, we leave it to the reader to see that there are no nontrivial Casnati-Ekedahl or Maroni loci.
   \end{asparaitem}

 \end{proof}
 
 As before, we now exhibit $U_{E^\gen, F^\gen}$ as a quotient. For
 brevity, set $E = E^\gen$ and $F = F^\gen$. Set
 \[ V := H^0(\P^1, \wedge^2 F \otimes E \otimes \det E).\]
An element $v \in \P_{\sub}V$ defines an anti-symmetric matrix as in \eqref{eqn:Matrix5}. Let $C_v$ be the zero locus of the $4 \times 4$ sub-Pfaffians of this matrix. Let $V^\circ \subset \P_{\sub}V$ be the open locus consisting of $v$ for which $C_v$ is irreducible and at worst nodal. Let $G_F := \Aut(\P F/\P^1)$ and $G_E := \Aut(\P E/\P^1$. Then $G := G_F \times G_E$ acts on $V^\circ$. The assignment $v \mapsto [\pi \from C_{v} \to \P^1]$ defines a map 
\[q \from V^\circ \to \td \H_{5,g}^\dagger.\]
Let $U^\dagger_{E, F}$ be the preimage of $U_{E^\gen, F^\gen}$ under $\td \H^\dagger_{5,g} \to \td \H_{5,g}$.
 \begin{proposition}\label{thm:5_quotient}
   The image of $q$ is $U^\dagger_{E^\gen, F^\gen}$. The fibers of $q$ consist of single $G$-orbits.
 \end{proposition}
  \begin{proof}
    The proof is exactly analogous to that of  \autoref{thm:3_quotient}.
  \end{proof}
 
  \begin{proposition}\label{thm:PRC5}
    [Picard rank conjecture for degree five]
    We have $\Pic_\Q \H_{5,g} = 0$.
  \end{proposition}
    \begin{proof}
    The proof is entirely analogous to the proof of
    \autoref{thm:PRC4}. We indicate only the major steps. Set $U =
    U_{E^\gen, F^\gen}$, and $U^\dagger =
    U^\dagger_{E^\gen,F^\gen}$. Applying \autoref{thm:pic_quotient} to
    $V^\circ \to U^\dagger$ and $U^\dagger \to U$, we get
    \[ \rk \Pic_\Q U \leq 1 + \rk\chi(G).\] Let $e$ be the number of
    divisorial components of $\td \H_{5,g} \setminus U$. We then get
    \[ \rk \Pic_\Q \td \H_{5,g} \leq 1 + \rk\chi(G) + e.\] Both $G$ and $e$ depend
    on $g$ modulo 4 and 5. Using \autoref{thm:5_compUe} and
    \autoref{thm:5_compUf}, we get the following possibilities.
    \[
    \begin{tabular}{c | c l}
      & $\qquad \rk\chi(G) = \rk\chi(G_E) + \rk\chi(G_F) \qquad$ & $e$ \\
      \hline
      $4 \mid g$, $5 \mid g+4$ & $0 = 0 + 0$ & 2 ($M$ and $CE$)\\
      $4 \mid g$, $5 \nmid g+4$& $1 = 0 + 1$ & 1 ($M$)\\
      $4 \nmid g$, $5 \mid g+4$& $1 = 1 + 0$ & 1 ($CE$)\\
      $4 \nmid g$, $5 \nmid g+4$& $2 = 1 + 1$& 0. 
    \end{tabular}
    \]
    In all the cases, we have $\Pic_\Q\td \H_{5,g} \leq 3$. With
    \autoref{thm:lin_indep_boundary}, this gives $\Pic_\Q \H_{5,g} =
    0$.
 \end{proof}

\section{From Hurwitz spaces to Severi varieties}
The associated scroll construction in \autoref{sec:ass_scroll} lets us relate the Picard rank of a Hurwitz space to the Picard rank of a Severi variety. In this section, we work out this relation.

Recall the notation $\U_g(\F_m, d\tau)$, $\V_g(\F_m, d\tau)$, and $\V^\irr_g(\F_m, d\tau)$ from \autoref{sec:notation}. When confusion is unlikely, we abbreviate them by $\U$, $\V$, and $\V^\irr$. Following Diaz and Harris \cite{harrisdiaz:geom_severi}, we enlarge $\U$ by including the irreducible curves of geometric genus $g$ having a cusp, a tacnode, a triple point, and irreducible nodal curves of geometric genus $(g-1)$ (that is, curves having an ``additional'' node). Denote by $\td \U$ the normalization of this partial compactification. The local analysis from \cite[\S~1]{harrisdiaz:geom_severi} of the Severi variety at points corresponding to cusps, tacnodes, triple points, and an additional node shows that $\td \U$ is smooth. Since $\td \U$ maps to the linear series $|d\tau|$, it carries over it a family of (singular) curves. The normalization of the total space of this family gives a family $\mathcal C \to \td \U$ of curves of arithmetic $g$. A generic fiber of $\mathcal C \to \td \U$ is the normalization the corresponding curve on $\F_m$.

Using the universal family, we can construct tautological divisor classes on $\td \U$ as follows. Consider the diagram 
\[
\begin{tikzpicture}[node distance=1.5cm]
  \node (C) {$\mathcal C$};
  \node [right of =C] (F) {$\F_m$};
  \node [below of = C] (U) {$\td{\U}$};
  \draw [->] 
  (C) edge node [left] {$\rho$} (U) 
  (C) edge node [above] {$\nu$} (F);
\end{tikzpicture}.
\]
Define five tautological divisor classes on $\td \U$ (The subscript $s$ stands for ``Severi''):
\begin{enumerate}
\item $\lambda_{s} := c_1(\rho_{*} \omega_\rho)$
\item $\kappa_{s} := \rho_{*}(c_1(\omega_{\rho})^2)$
\item $\xi_{s} := \rho_{*}(\nu^{*}(f) \cdot c_1(\omega_{\rho}))$
\item $\theta_{s} := \rho_{*}(\nu^{*}(\sigma) \cdot c_1(\omega_{\rho}))$
\item $\psi_{s} := \rho_{*}(\nu^{*}\text{[Point]})$
\end{enumerate}
Since the irreducible curves in the linear system $|d\tau|$ avoid the directrix $\sigma$, we get $\theta_s = \psi_s = 0$. Therefore, a natural conjecture is the following.
\begin{conjecture}\label{conjectureS} 
  The rational Picard group of $\td{\U}$ is tautological, that is 
  \[\Pic_{\Q}\td{\U} = \Q \langle  \lambda_{s}, \kappa_{s}, \xi_{s} \rangle.\]
\end{conjecture}
Denote by $CU$, $TN$, $TP$, and $\Delta$ the closures in $\V^\irr$ of the locus curves with a cusp, tacnode, triple point, or an additional node, respectively. Abusing notation, denote their preimages in $\td \U$ by the same letters.

\begin{remark}\label{rem:equivalenceS}
  It is not hard to check that the classes in $\Pic_\Q\td \U$ of $CU$, $TN$, $TP$, and $\Delta$ can be expressed as $\Q$-linear combinations of $\lambda_s$, $\kappa_s$, and $\xi_s$ and vice versa. \autoref{conjectureS} is therefore equivalent to
  \[ \Pic_\Q \U = 0.\]
\end{remark}

\begin{proposition}\label{sev:divisors}
  The only divisorial components of $\V^\irr \setminus  \U$ are $CU$, $TN$, $TP$, and $\Delta$.
\end{proposition}
\begin{proof}
  It suffices to show that the codimension one components of $\V \setminus \U$ are the loci of curves with cusps, tacnodes, triple points or an additional node. This follows by the same proof as for Theorem~1.4 in \cite{diaz_harris_ideals}. The critical ingredient of the argument is provided by \autoref{sev:adjointconditions}.
\end{proof}
\begin{lemma}\label{sev:adjointconditions}
  Let $D \in |d\tau|$ be a reduced irreducible curve on the Hirzebruch surface $\F_{m}$. Denote by $A$ the conductor ideal of the singularities of $D$. Then $A$ imposes independent conditions on $H^{0}(\F_{m}, O(d\tau))$.
\end{lemma}
\begin{proof}
  Let $K = K_{\F_m}$ be the canonical class. The anti-canonical class $-K$ is effective. Furthermore, the fixed component of $-K$ is the directrix $\sigma$, and $-K$ separates points away from $\sigma$. 

  It is a classical that $A$ imposes independent conditions on the adjoint linear system $|K+D|$. Let $Z = V(A)$ be the zero dimensional scheme defined by the ideal sheaf $A$. Then the restriction map \[H^{0}(O(K+D)) \to H^{0}(O_{Z}(K+D))\] is surjective.  Therefore, we can conclude the same for \[H^{0}(O(D)) \to H^{0}(O_{Z}(D))\] by multiplying the previous restriction map by a general section of $O(-K)$.
\end{proof}

We now rephrase the Picard rank conjecture for Hurwitz spaces in a manner similar to \autoref{conjectureS}. Consider the diagram
\[
\begin{tikzpicture}[node distance=1.5cm]
  \node (C) {$\mathcal C$};
  \node [right of =C] (P) {$\P^1$};
  \node [below of = C] (H) {$\td \H^\dagger_{d,g}$};
  \draw [->] 
  (C) edge node [left] {$f$} (H) 
  (C) edge node [above] {$\alpha$} (P);
\end{tikzpicture}.
\]
Define the following tautological divisor classes on $\td \H_{d,g}^\dagger$ (The subscript ``h'' stands for ``Hurwitz''):
\begin{enumerate}
\item $\lambda_{h} := c_1(f_{*} \omega_f)$
\item $\kappa_{h} := f_{*}(c_1(\omega_{f})^2)$
\item $\xi_{h} := f_{*}(\alpha^{*}\text{[Point]} \cdot c_1(\omega_{f}))$
\end{enumerate}

\begin{conjecture}\label{conjectureH} 
  The rational Picard group of $\td{\H}^\dagger_{d,g}$ is tautological, that is,
  $\Pic_{\Q}\td{\H}^\dagger_{d,g} = \Q \langle \lambda_{h},
  \kappa_{h}, \xi_{h} \rangle.$
\end{conjecture}
\begin{remark}\label{rem:equivalenceH}
  It is easy to see that the classes of $T$, $D$, and $\Delta$ can be
  expressed as $\Q$-linear combinations of $\lambda_h$, $\kappa_h$,
  and $\xi_h$ and vice versa. Also, by \autoref{thm:framed_unframed},
  the framed/unframed distinction is irrelevant. Therefore,
  \autoref{conjectureH} is equivalent to the Picard rank conjecture
  stated in the introduction, namely that
  \[ \Pic_\Q \H_{d,g} = 0.\]
\end{remark}

We now state the main theorem of this section.
\begin{theorem}\label{SvsH}
  If $m \geq \lfloor (g+d-1)/(d-1) \rfloor$, then \autoref{conjectureS} for $\td \U_g(\F_m, d\tau)$ implies \autoref{conjectureH} for $\td{\H}_{d,g}^{\dagger}$. If $m \geq \lceil 2(g+d-1)/d \rceil$, then \autoref{conjectureS} for $\td{\U}_g(\F_m, d\tau)$  is equivalent to \autoref{conjectureH} for $\td{\H}_{d,g}^{\dagger}$. 
\end{theorem}

\begin{proof}
  Let $m \geq \lfloor (g+d-1)/(d-1)\rfloor$. Retain the notation introduced in this section. In particular, abbreviate $\U_g(\F_m, d\tau)$ by $\U$, and so on. Let $\pi \from \F_m \to \P^1$ be the projection and $\sigma \subset \F_m$ the directrix. Fix a section $\zeta \in H^0(\F_m, \pi^*O(m))$ corresponding to a smooth element of the linear series $|\tau|$. We view $\F_m \setminus \sigma$ as the total space of the line bundle $O(m)$ on $\P^1$ and $\zeta$ as the tautological section of $\pi^*O(m)$ on this total space.

Let $\phi \from \mathcal C \to \P^1$ be the composite $\phi = \pi \circ \nu$. Let $Z \subset \td \U$ be the open subset consisting of $u$ where $h^0(\mathcal C_u, \phi^*O(m))$ is minimal. Likewise, let $W \subset \td \H_{d,g}^\dagger$ be the subset consisting of $[\alpha \from C \to \P^1]$ where $h^0(C, \alpha^*O(m))$ is minimal. By \autoref{maronidivisor}, the complement of $W$ in $\td \H_{d,g}^\dagger$ has codimension at least two. Let $V$ be the total space of the vector bundle $f_* \alpha^*O(m)|_W$ over $W$.

  We have a birational morphism $q \from Z \to V$ defined as follows. A point $u \in \td \U$ is mapped to $[\phi_u \from \mathcal C_u \to \P^1, v]$, where $v \in H^0(\mathcal C_u, \phi_u^*O(m))$ is the restriction of $\zeta$. To define the inverse, we must restrict to an open subset of $V$. Let $X \subset V$ be the open subset consisting of $([\alpha \from C \to \P^1], v)$, where $v \in H^0(C, \alpha^*O(m))$ is such that the lift of $C \to \P^1$ to $C \to \F_m$ defined by $v$ is birational onto its image. We then get a morphism $p \from X \to \V^\irr$, which is quasi-finite, and generically one-to-one. Let $Y \subset X$ be the open subset consisting of points whose associated element in $\V^\irr$ has at worst a cusp, a tacnode, a triple point, or an additional node. By \autoref{sev:divisors}, and the quasi-finiteness of $p$, the complement of $Y$ in $X$ has codimension at least two. Since $Y$ is normal, we get a morphism $p \from Y \to Z \subset \td \U$, inverse to $q$. We summarize the spaces we have defined and their relationships in the following diagram.
\begin{equation}\label{eqn:picture}
\begin{tikzpicture}
  \node (H) {$\td \H_{d,g}^\dagger$};
  \node (W) [right=2em of H] {$W$}; 
  \node (V) [above of=W]{$V$};
  \node (X) [right of=V]{$X$};
  \node (Y) [right of=X]{$Y$};
  \node (Z) [right of=Y]{$Z$};
  \node (U) [below of=Z]{$\td\U$};
  \draw[left hook->] 
  (W) edge node[above]{$\star$} (H) 
  (X) edge (V) 
  (Y) edge node[above]{$\star$} (X)
  (Z) edge (U);
  \draw[->] (V) edge node[right]{$\star$} (W);
  \draw[->,bend left=20] (Y) edge node[above]{\small $p$} (Z);
  \draw[->,bend left=20] (Z) edge node[below]{\small $q$} (Y);
\end{tikzpicture}.
\end{equation}
The inclusions are open inclusions. $Y$ and $Z$ are isomorphic via $p$ and $q$. The maps marked by $\star$ induce isomorphisms on Picard groups. For the open inclusions, this is because the complements have codimension at least two. For $V \to W$, this is because it is a vector bundle. 

Denote the pullbacks of $\lambda_h$, $\kappa_h$, and $\xi_h$ to $W$, $V$, $X$, and $Y$ by the same letters. Then, we have
\begin{align*}
  p^* \lambda_s &= \lambda_h &  p^* \kappa_s &= \kappa_h & p^* \xi_s &= \xi_h\\
  q^* \lambda_h &= \lambda_s  & q^* \kappa_h &= \kappa_s & q^* \xi_h &= \xi_s.
\end{align*}
We may thus drop the subscripts and use $\lambda$, $\kappa$, and $\xi$ to denote the corresponding divisors on any of the spaces in \eqref{eqn:picture}. 

Before we proceed, we must comment on the inclusion $X \hookrightarrow V$. The complement consists of $([\alpha \from C \to \P^1], v)$, where $v \in H^0(C, \alpha^*O(m))$ does not give a birational map to $\F_m$. Let us disregard the $\alpha$'s that factor non-trivially (such $\alpha$'s form set of codimension at least two). Then the only such $v$ are the pullbacks of the sections in $H^0(\P^1, O(m))$. The locus $([\alpha \from C \to \P^1], v)$, where $v \in \alpha^*H^0(\P^1, O(m))$ has codimension at least two except in the case $g \equiv -1 \pmod {(d-1)}$, and $m = \lfloor (g+d-1)/(d-1)\rfloor$, that is, when the generic splitting of $\alpha_* O_C$ is
\[ \alpha_* O_C = O \oplus O(-m) \oplus O(-m-1) \oplus \dots \oplus O(-m-1).\]
In this case, the complement of $X$ in $V$ has a divisorial component given by the image of the constant vector bundle $H^0(\P^1, O(m)) \otimes O_W$. However, the class of this divisor in $\Pic_\Q V \cong \Pic_\Q W$ is in the span of $\lambda$, $\kappa$, and $\xi$. Therefore, in any case, $\Pic_\Q V$ is spanned by $\lambda$, $\kappa$, and $\xi$ if and only if $\Pic_\Q X$ is.

Assume that \autoref{conjectureS} holds. From diagram \eqref{eqn:picture}, we see that $\Pic_\Q X$ is spanned by $\lambda$, $\kappa$, and $\xi$. By the comment about $X \hookrightarrow V$ above, this implies that $\Pic_\Q V$, and in turn $\Pic_\Q \td \H^\dagger_{d,g}$ is spanned by $\lambda$, $\kappa$, and $\xi$. Hence \autoref{conjectureH} holds.

Assume that $m \geq \lceil 2(g+d-1)/(d-1) \rceil$ and \autoref{conjectureH} holds. Then, by \autoref{thm:constraintsE} the inclusion $Z \hookrightarrow \td \U$ is in fact an isomorphism. Again, diagram \eqref{eqn:picture} shows that $\Pic_\Q \td \U$ is spanned by $\lambda$, $\kappa$, and $\xi$. Hence \autoref{conjectureS} holds.
\end{proof}

\section*{Acknowledgments}
We are grateful to our advisor Joe Harris for suggesting this problem
and for providing constant support and encouragement. We also thank
Gabriel Bujokas, Dawei Chen, Maksym Fedorchuk, and Ravi Vakil for
helpful comments and conversations.  \bibliographystyle{amsalpha}
\bibliography{CommonMath}
\end{document}